\documentclass{article}

\usepackage{amsmath}
\usepackage{amssymb}
\usepackage{latexsym}
\usepackage{color}

\newtheorem{assumption}{Assumption}
\def\qed{ \ \vrule width.2cm height.2cm depth0cm\smallskip}
\newenvironment{proof}{\noindent {\bf Proof.\/}}{$\qed$\vskip 0.1in}

\newcommand{\we}{\wedge}
\newcommand{\ol}{\overline}
\newcommand{\ul}{\underline}

\newcommand{\ba}{\begin{array}}
\newcommand{\ea}{\end{array}}
\newcommand{\be}{\begin{equation}}
\newcommand{\ee}{\end{equation}}
\newcommand{\bea}{\begin{eqnarray}}
\newcommand{\eea}{\end{eqnarray}}
\newcommand{\beaa}{\begin{eqnarray*}}
\newcommand{\eeaa}{\end{eqnarray*}}

\def\dbE{\mathbb{E}}
\def\dbF{\mathbb{F}}

\def\dbI{\mathbb{I}}

\def\dbL{\mathbb{L}}

\def\dbN{\mathbb{N}}
\def\dbP{\mathbb{P}}
\def\dbR{\mathbb{R}}
\def\dbS{\mathbb{S}}

\def\dbQ{\mathbb{Q}}

\def\a{\alpha}
\def\b{\beta}
\def\g{\gamma}
\def\d{\delta}
\def\e{\varepsilon}

\def\k{\kappa}
\def\l{\lambda}

\def\si{\sigma}
\def\t{\tau}

\def\th{\theta}
\def\o{\omega}

\def\Th{\Theta}

\def\F{\Phi}
\def\O{\Omega}
\def\cE{{\cal E}}
\def\cF{{\cal F}}

\def\cJ{{\cal J}}

\def\cL{{\cal L}}

\def\cP{{\cal P}}

\def\cT{{\cal T}}

\def\ch{\textsc{h}}

\def\no{\noindent}

\def\ms{\medskip}

\def\q{\quad}

\def\pa{\partial}
\def\cd{\cdot}
\def\cds{\cdots}

\def\qed{ \hfill \vrule width.25cm height.25cm depth0cm\smallskip}

\newcommand{\basa}{\begin{assumption}}
\newcommand{\easa}{\end{assumption}}

\newcommand{\bas}{\begin{assum}}
\newcommand{\eas}{\end{assum}}

\def\limsup{\mathop{\overline{\rm lim}}}
\def\liminf{\mathop{\underline{\rm lim}}}

\def\pa{\partial}

 \def\cd{\cdot}
\def\cds{\cdots}

\def\1{{\bf 1}}

\def\:{\!:\!}

\def \proof{{\noindent \bf Proof\quad}}

at 9pt

\begin{document}

\newtheorem{thm}{Theorem}[section]
\newtheorem{lem}[thm]{Lemma}
\newtheorem{cor}[thm]{Corollary}
\newtheorem{prop}[thm]{Proposition}
\newtheorem{rem}[thm]{Remark}
\newtheorem{eg}[thm]{Example}
\newtheorem{defn}[thm]{Definition}
\newtheorem{assum}[thm]{Assumption}

\renewcommand {\theequation}{\arabic{section}.\arabic{equation}}
\def\thesection{\arabic{section}}

\numberwithin{equation}{section}
\numberwithin{thm}{section}

\title{\bf Comparison of viscosity solutions of fully nonlinear degenerate parabolic Path-dependent PDEs}

\author{Zhenjie  {\sc Ren}\footnote{CMAP, Ecole Polytechnique Paris, ren@cmap.polytechnique.fr.}   \and
Nizar {\sc Touzi}\footnote{CMAP, Ecole Polytechnique Paris, nizar.touzi@polytechnique.edu. Research supported by ANR the Chair {\it Financial Risks} of the {\it Risk Foundation} sponsored by Soci\'et\'e G\'en\'erale, and the Chair {\it Finance and Sustainable Development} sponsored by EDF and Calyon. }
       \and Jianfeng {\sc Zhang}\footnote{University of Southern California, Department of Mathematics, jianfenz@usc.edu. Research supported in part by NSF grant DMS 1413717.}}

\maketitle

\begin{abstract}
We prove a comparison result for viscosity solutions of (possibly degenerate) parabolic fully nonlinear path-dependent PDEs. In contrast with the previous result in Ekren, Touzi \& Zhang \cite{ETZ-part2}, our conditions are easier to check and allow for the degenerate case, thus including first order path-dependent PDEs. Our argument follows the regularization method as introduced by Jensen, Lions \& Souganidis \cite{JLS} in the corresponding finite-dimensional PDE setting. The present argument significantly simplifies the comparison proof in \cite{ETZ-part2}, but requires an $\mathbb{L}^p-$type of continuity (with respect to the path) for the viscosity semi-solutions and for the nonlinearity defining the equation. 
\end{abstract}

\section{Introduction}

This paper provides a proof for the comparison result for viscosity solutions of the fully nonlinear path dependent partial differential equation:
\bea\label{PPDE-intro}
-\pa_t u(t,\o) - G\big(t,\o, u(t,\o),\pa_\o u(t,\o), \pa^2_{\o\o} u(t,\o)\big) \le  0
 &\mbox{on}&
 [0,T)\times\Omega.
 \eea
Here, $T>0$ is a given terminal time, and $\o\in \O$ is a continuous path from $[0,T]$ to $\dbR^d$ starting from the origin. The nonlinearity $G$ is a mapping from $ [0,T]\times\O\times\dbR\times\dbR^d\times\dbS^d$ to $\dbR$, where $\dbS^d$ denotes the set of all $d\times d$-symmetric matrices.
 
Such equations arise naturally in many applications. For instance, the dynamic programming equation (also called Hamilton-Jacobi-Bellman equation) associated to a problem of stochastic control of non-Markov diffusions falls in the class of equations \eqref{PPDE-intro}, see \cite{ETZ-part1}. In particular hereditary control problems may be addressed in this context rather than embedding the problem into a PDE on the Hilbert space $\dbL^2([0,T])\supset\Omega$, see \cite{CossoFedericoGozziRosestolatoTouzi}. Similarly, stochastic differential games with non-Markov controlled dynamics lead to path-dependent Isaac-Hamilton-Jacobi-Bellman equations of the form \eqref{PPDE-intro}, see \cite{PhamZhang}. The notion of nonlinear path dependent partial differential equation was first proposed by Peng \cite{Peng-ICM}. A crucial tool to study such equation is the functional It\^{o} calculus, initiated by Dupire \cite{Dupire}, and further studied by Cont \& Fournie \cite{CF}. We also refer to Peng \& Wang \cite{PW} for some study on classical solutions of semilinear equations. 
 
The notion of viscosity solutions studied in this paper was introduced by Ekren, Keller, Touzi \& Zhang \cite{EKTZ} in the semilinear context, and further extended to the fully nonlinear case by Ekren, Touzi \& Zhang \cite{ETZ-part1, ETZ-part2}. Following the lines of the classical Crandall \& Lions \cite{CrandallLions} notion of viscosity solutions, supersolutions and subsolutions are defined through tangent test functions. However, while Crandall \& Lions consider pointwise tangent functions, the tangency conditions in the present path-dependent setting is in the sense of the mean with respect to an appropriate class of probability measures $\cP$. In particular, when restricted to the Markovian case, our notion of viscosity solutions involves a larger set of test functions. This is in favor of uniqueness but may make the existence issue more difficult. We refer to Ren, Touzi \& Zhang \cite{RTZ-survey} for an overview.

Throughout this paper, the notion of $\cP-$viscosity solution refers that introduced in \cite{EKTZ, ETZ-part1}. For the sake of clarity, the classical notion of viscosity solutions based on pointwise tangent test functions will be sometimes referred to as the Crandall-Lions notion of viscosity solution.

The wellposedness of the notion of $\cP-$viscosity solutions was first proved in \cite{EKTZ} in the semilinear case, and later extended to the fully nonlinear case in \cite{ETZ-part2}. In contrast with the classical wellposedness theory for the Crandall-Lions viscosity solutions in finite dimensional spaces, the comparison and existence results proved in \cite{EKTZ,ETZ-part2} are inter-connected. Moreover, the proof relies heavily on the corresponding finite-dimensional PDE results applied to path-frozen versions of \eqref{PPDE-intro}, and thus does not really take advatange of the larger class of test functions. Finally, the technical conditions of \cite{EKTZ,ETZ-part2} exclude the degenerate case. In particular, their main wellposedness result can not be viewed as an extension of Lukoyanov \cite{Lukoyanov},  where the author studied the wellposedness of the viscosity solutions to the first order path-dependent PDEs.

In our recent paper \cite{RTZ} we provided a purely probabilistic comparison proof in the semilinear setting, which is completely disconnected from the existence result and which allows for degenerate equations. The importance of a separate proof of comparison was highlighted in the Crandall-Lions theory of viscosity solutions: it allows access to the Perron existence argument, and was shown to play an important role in the regularity of viscosity solutions, and in the convergence of numerical approximations together with the analysis of the order of the corresponding error. Similar to the classical finite-dimensional theory of viscosity solutions, the (disconnected) comparison result in \cite{RTZ} opens the door for an existence argument by the so-called Perron method, see Ren \cite{Ren-perron}, or by a limiting argument \`a la Barles \& Souganidis \cite{BarlesSouganidis}, see Zhang \& Zhuo \cite{ZhangZhuo} and Ren \& Tan \cite{RenTan}.

Our argument in \cite{RTZ} was crucially based on an adaptation of the Caffarelli \& Cabre \cite{CaffarelliCabre} notion of punctual differentiation to our setting, namely the $\cP-$punctual differentiation. In particular, denoting by $\dbP_0$ the Wiener measure on the space of continuous paths, we have reported in \cite{RTZ} an easy proof of the equivalence between our notion of $\{\dbP_0\}-$viscosity subsolution of the heat equation and the submartingale property. This equivalence implies an easy proof of the comparison result for linear path-dependent PDEs, thus highlighting the importance of enlarging the set of test functions under the notion of $\{\dbP_0\}-$viscosity solutions. The semilinear case is more involved, but uses standard stochastic analysis arguments. A crucial (and surprising) result obtained in \cite{RTZ} is that all $\cP-$viscosity subsolutions with appropriate integrability are $\cP-$punctually differentiable Leb$\otimes\dbP_0-$ almost everywhere on $[0,T]\times\Omega$. We recall that in the finite-dimensional context, the punctual differentiability is satisfied by an appropriate approximation of the subsolution. 
 
The main contribution of this paper is to provide a comparison result for the viscosity solutions of fully nonlinear path-dependent PDEs \eqref{PPDE-intro} which does not involve the existence issue. Our result is established under general conditions on the nonlinearity. Namely, we establish the comparison between $d_p-$uniformly continuous subsolutions and supersolutions under the conditions that the nonlinearity $G$ is 

- $d_p-$uniformly continuous in $\theta$, uniformly in $(y,z,\g)$, 

- and Lipschitz-continuous in $(y,z,\g)$, uniformly in $\theta$.

\no The last conditions represent a significant simplification of the assumptions required in Ekren, Touzi \& Zhang \cite{ETZ-part2}. Moreover, we emphasize that our conditions allow for degenerate parabolic equations, and therefore contain the first order path-dependent Hamilton-Jacobi equations of Lukoyanov \cite{Lukoyanov}. 

 Our arguments are inspired from the work of Jensen, Lions \& Souganidis \cite{JLS}, in which one of the main ideas is to find approximations of viscosity sub- and supersolutions of PDEs. The approximations proposed in \cite{JLS} are due to Lasry \& Lions \cite{LL}. Let $u$ be a viscosity subsolution, and $u^n$ be the approximations. A careful examination of their proof shows that good approximations should in general satisfy:

\vspace{3mm}

\noindent (A1) $\lim_{n\rightarrow \infty} u^n = u$;
\\
(A2) $u^n$ are more regular than $u$ (thus we may call the approximations as regularization);
\\
(A3) $u^n$ are still viscosity subsolutions for some equations approximating the original one.

\vspace{3mm}

A similar regularization was introduced in Ren \cite{Ren-perron}, for functions in the path space, in order to study the comparison of semi-continuous viscosity solutions to semilinear path-dependent PDEs. However such a regularization fails to achieve the purpose of the present paper.

As a key technical tool in this paper, we introduce a new regularization for viscosity sub- and supersolutions in the  context of fully nonlinear path-dependent PDEs, which allows to prove the final comparison result. Constrained by our method, we are unfortunately not able to compare the viscosity sub- and supersolutions which are continuous in the (pseudo-)distance $d_\infty$,
\beaa
d_\infty\big((t,\o),(t',\o')\big) &=& |t-t'| + \|\o_{t\we\cd} - \o'_{t'\we\cd}\|_\infty,
\eeaa
used in the previous works on path dependent PDEs. Instead, we prove the comparison for viscosity sub- and supersolutions which are continuous in the sense of the following (pseudo-)distance: 
 \beaa
 d_p\big((t,\o),(t',\o')\big) ~=~ |t-t'| + \|\o_{t\we\cd} - \o'_{t'\we\cd}\|_p^p,
 \q\mbox{where}\q \|\o\|_p^p = \int_{0}^{T+1} |\o_t|^p dt.
 \eeaa
This continuity is slightly stronger than that under $d_\infty$. However, since $\|\o\|_p\rightarrow \|\o\|_\infty$, our continuity requirement can be viewed as a slight strengthening of the $d_\infty-$continuity. In order to justify the relevance of the $d_p-$continuity, we provide in this paper a large class of path-dependent fully nonlinear equations with unique $d_p-$continuous viscosity solutions. This is achieved by complementing our comparison result with an example of stochastic control problem whose value function is a $d_p-$continuous $\cP-$viscosity solution of the corresponding path-dependent dynamic programming equation.

Given a $\cP-$viscosity solution $u$, the regularization introduced in the present paper defines functions $u^n$ satisfying the above requirement (A1). However, rather than verifying (A2) and (A3) in the sense of $\cP-$viscosity solutions, we show that $u^n$ induces a continuous finite-dimensional function  which is a viscosity solution of an appropriate PDE in the classical sense of Crandall-Lions. We recall that this is a weaker conclusion than the corresponding notion of $\cP-$viscosity solutions. This allows to reduce the comparison task to the well-established notion of viscosity solutions in the finite-dimensional context, and represents the major difference with the approach used in our previous paper \cite{RTZ} focused on the semilinear case.

The rest of the paper is organized as follows. Section \ref{sec:notation} introduces the main notations. Section \ref{sec:preli} recalls some useful results from the previous work on path dependent PDEs. Section \ref{sec:mainresult} states the main assumptions and results of this paper. In Section \ref{sect:regularization} we introduce the regularization, and prove its main properties. Further, in Section \ref{sect:comparison} we use the regularization to prove the comparison result. Finally, Section \ref{sec:appendix} concludes the paper by an example of a stochastic control problem whose value function is a $d_p$-continuous $\cP-$viscosity solution of the corresponding path-dependent dynamic programming equation, under natural assumptions on the ingredients of the control problem.

\section{Notations}\label{sec:notation}

Throughout this paper let $T>0$ be a given finite maturity, $\O:=\{\o\in C([0,T];\dbR^d):\o_0=0\}$ be the set of continuous paths starting from the origin, and $\Theta:=[0,T]\times\O$. For the convenience of notation, we often denote by $\th$ the pair $(t,\o)$. We denote by $B$ the canonical process on $\O$, by $\dbF = \{\cF_t, 0\le t\le T\}$ the canonical filtration,  by $\dbP_0$ the Wiener measure on $\O$, and by $\cT$ the set of all $\dbF$-stopping times taking values in $[0,T]$. Further, for $\ch \in \cT$, denote by $\cT_\ch$ the subset of $\t\in \cT$ taking values in $[0, \ch]$.

For $\o,\o'\in\O$ and $t\in [0,T]$, we define
 $$
 (\o\otimes_t\o')_s 
 := 
 \o_s\1_{\{s< t\}}
 +(\o_t+\o'_{s-t})\1_{\{s\ge t\}}.
 $$ 
Let $\xi:\O\rightarrow \dbR$ be $\cF_T$-measurable random variable. For any $\th=(t,\o) \in\Theta$, define
 \beaa
 \xi^\th (\o')
 :=
 \xi\big(\o\otimes_t\o'\big)
 &\mbox{for all}&
 \o'\in\O.
 \eeaa
Clearly, $\xi^{\th}$ is $\cF_{T-t}$-measurable, and thus $\cF_T$-measurable. Similarly, given a process $X$ defined on $\O$, we denote:
$$X^\th_s(\o^{'}):=X_{t+s}(\o\otimes_t\o^{'}),\ \text{for }s\in[0,T-t],~\o'\in\O.$$
Clearly, if $X$ is $\dbF$-adapted, then so is $X^\th$.

As in Ekren, Touzi \& Zhang \cite{ETZ-part2}, for every constant $L>0$, we denote by $\cP_L$ the collection of all continuous semimartingale measures $\dbP$ on $\O$ whose drift and diffusion are bounded by constant $L$, respectively. More precisely, let $\tilde\O:=\O\times\O\times\O$ be an enlarged canonical space, $\tilde B:=(B,A,M)$ be the canonical process. A probability measure $\dbP\in\cP_L$ means that there exists an extension $\dbQ^{\a,\b}$ of $\dbP$ on $\tilde \O$ such that:
\be\label{defn:PL}
\left.\ba{lll}
&B=A+M,\q A~\mbox{is absolutely continuous,}~M~\mbox{is a martingale,}&\\
&\|\a^\dbP\|_\infty, ~\|\b^\dbP\|_\infty ~\le ~ L,\q
\mbox{where}~\a^\dbP_t:=\frac{dA_t}{dt},~\b^\dbP_t:=\sqrt{\frac{d\langle M\rangle_t}{dt}},&
\ea\right.\dbQ^{\a,\b}\mbox{-a.s.}
\ee
 We also introduce the sublinear and superlinear expectation operators associated to $\cP_L$:
 \beaa
 \overline{\cE}_L
 :=\sup_{\dbP\in\cP_L} \dbE^\dbP
 &\mbox{and}&
 \underline{\cE}_L
 :=\inf_{\dbP\in\cP_L} \dbE^\dbP.
 \eeaa
 One may easily prove the following lemma.
 
 \begin{lem}\label{lem:EH}
There is a constant $C>0$ such that we have for all $\dbP\in \cP_L$ and $\ch\in \cT$ that
 \beaa
 \big| \dbE^\dbP [B_\ch] \big| \le C\dbE^\dbP [\ch]
 \q \mbox{and} \q
 \dbE^\dbP \big[ |B_\ch|^2 \big] \le C\dbE^\dbP [\ch]. 
 \eeaa
 \end{lem}

In this paper, we consider a new $\dbL^p$-type of distance in the space $\Th$.

\begin{defn}
For $p\ge 1$, we introduce the following distance for the space $\Th$:
\beaa
d_p (\th,\th') ~:=~ |t-t'| + \| \o_{t\we\cd} - \o'_{t'\we\cd}\|_p , \q\mbox{for all}~\th,\th'\in \Th,
\eeaa
where 
\be\label{eq:deltap}
\| \o \|_p^p ~:=~  |\o_T|^p+\int_0^T |\o_s |^p ds,\q\mbox{for all}~ \o,\o'\in \O. 
\ee
We say that a function $f: \Th\rightarrow \dbR$ is $d_p$-continuous, if $f$ is continuous with respect to $d_p(\cd,\cd)$.
\end{defn}

\begin{rem}{\rm
Let $d_\infty(\cd,\cd)$ be the distance between continuous paths introduced by Dupire \cite{Dupire}, i.e.
\beaa
d_\infty (\th,\th') ~:=~ |t-t'| + \max_{0\le s\le T} |\o_{s\we t} -\o'_{s\we t'}|.
\eeaa
Note that
\beaa
d_p (\th,\th') ~\le ~ C d_\infty (\th,\th')
\q\mbox{and}\q
\lim_{p\rightarrow\infty} d_p (\th,\th') ~= ~  d_\infty (\th,\th').
\eeaa
In particular, a $d_p$-continuous function is automatically continuous in Dupire's sense.
}
\end{rem}

For later use, we observe that
\be\label{eq:stopint}
\|\o_{t\we\cd}\|_p^p ~ = ~ \int_0^{T+1} |\o_{t\we s}|^p ds ~ = ~  (T+1 -t) |\o_t|^p+\int_0^t |\o_s|^p ds.
\ee

Example \ref{eg:L1continuous} in Appendix provides sufficient conditions for the value function of a stochastic control problem to be uniformly $d_p$-continuous. On the other hand, as discussed in Ekren, Touzi \& Zhang \cite{ETZ-part1}, the value function of a stochastic control problem can be proved to be a viscosity solution of the corresponding path dependent Hamilton-Jaccobi-Bellman equation. Therefore, there are many examples of fully nonlinear path dependent PDEs which have uniformly $d_p$-continuous viscosity solutions. In this paper, we focus on the uniqueness of such solutions.

We would like to emphasize that, throughout this paper, $C$ denotes a generic constant, which may change from line to line. For example the reader may find $2C\le C$, without any contradiction as the left-hand side $C$ is different from the right-hand side $C$.

\section{Preliminaries}\label{sec:preli}
Consider the fully nonlinear parabolic PDE:
 \bea\label{eq: PDE}
 -\pa_t u - g(t, x, u, D u, D^2 u) ~=~ 0
 &t<T,&
 x\in\dbR^d,
 \eea
where $\partial_t$ denotes the time derivative, and $Du,D^2u$ denote the space gradient and Hessian, respectively.

We first recall the definition of the classical Crandall-Lions viscosity solutions for parabolic PDEs. For $\a\inÊ\dbR$, $\b\in\dbR^d$, $\g\in\dbS^d$, define the paraboloid $\psi^{\a,\b,\g}$:
\beaa
\psi^{\a,\b,\g}(t,x) := \a t + \b\cd x +\frac12 x^{\rm T} \g x,~~t\ge 0,~x\in\dbR^d.
\eeaa
Then define the jets:
 \beaa
 \ul J u(t,x) 
 &:=& 
 \Big\{(\a,\b,\g):\q u(t,x) ~\ge ~ u(s,y) -\psi^{\a,\b,\g}(s-t,y-x) +o(|s-t|,|y-x|^2) \Big\},
 \\
 \ol J v(t,x) 
 &:=& 
 \Big\{(\a,\b,\g):\q v(t,x) ~\le ~ v(s,y) -\psi^{\a,\b,\g}(s-t,y-x) +o(|s-t|,|y-x|^2)  \Big\}.
\eeaa
We say that function $u$ is a viscosity subsolution of the PDE \eqref{eq: PDE}, if 
\beaa
-\a - g (t,x,u, \b,\g) ~\le~0, \q\mbox{for all}~(t,x)\in (0,T)\times\dbR^d~\mbox{and} ~
(\a,\b,\g)\in \ul J u(t,x).
\eeaa
Similarly, a function $v$ is a viscosity supersolution of the PDE \eqref{eq: PDE}, if 
\beaa
-\a - g (t,x,v, \b,\g) ~\ge~0, \q\mbox{for all}~(t,x)\in (0,T)\times\dbR^d~\mbox{and} ~
(\a,\b,\g)\in \ol J v(t,x).
\eeaa

In this paper, we consider the fully nonlinear parabolic path dependent PDE:
\be\label{eq: PPDE}
-\pa_t u - G(\th , u, \pa_\o u, \pa^2_{\o\o} u) ~=~ 0.
\ee
In our previous work \cite{RTZ-survey,RTZ} on viscosity solutions of path-depedent PDEs, it was already understood that one can define viscosity solutions via jets. For simplicity, this paper starts directly from this definition as it avoids to introduce the notion of smooth processes (i.e. those processes which satisfy an It\^o formula simultaneously under all probability measures $\dbP\in\cP_L$). Let
\beaa
 \phi^{\a,\b,\gamma}(\th)
 :=
 \a t+\b\cdot\omega_t+\frac12\o_t^{\rm T} \gamma\o_t,
 \q \th\in\Th,
 \eeaa
for some $(\a,\b,\g)\in\dbR\times\dbR^d\times\dbS_d$. We then introduce the corresponding subjet and superjet:
 \beaa
 \underline{\cJ}_L u(\th)
 &:=&
 \big\{ (\a,\b,\g)\in\dbR\times\dbR^d\times\dbS_d: 
          u(\th) ~=~ \max_{\t\in \cT_{\ch_\d}} \ol\cE_L\big[(u^{\th} - \phi^{\a,\b,\g})(\t,B)\big],~\mbox{for some}~\d>0
 \big\},
 \\
 \overline{\cJ}_{\!\!L} v(\th)
 &:=&
 \big\{ (\a,\b,\g)\in\dbR\times\dbR^d\times\dbS_d: 
            v(\th) ~=~ \min_{\t\in \cT_{\ch_\d}} \ul\cE_L\big[(v^{\th} - \phi^{\a,\b,\g})(\t,B)\big],~\mbox{for some}~\d>0
 \big\},
 \eeaa
 where
\be\label{eq:Heps} 
 \ch_\d := \d \we \inf\{t\ge 0: |B_t|\ge \d \}
 \ee
 is a stopping time, and is called the localization. 

\begin{defn}
 A function $u:[0,T]\times\O\longrightarrow\dbR$ is a 
 \\
 $\bullet$ $\cP_L$-viscosity subsolution of \eqref{eq:  PPDE} if
 $- \a - G(\th,u(\th),\b,\gamma)\le 0$ for all $\th\in \Theta,$ $(\a,\b,\gamma)\in\underline{\cJ}_L u(\th)$;
 \\
 $\bullet$ $\cP_L$-viscosity supersolution of \eqref{eq:  PPDE} if
 $- \a - G(\th,u(\th),\b,\gamma)\ge 0$ for all $\th\in \Theta,$ $(\a,\b,\gamma)\in\overline{\cJ}_L u(\th)$.
\end{defn}

It was proved in Ren, Touzi \& Zhang \cite{RTZ-survey} that this definition is equivalent to the original definition of Ekren, Touzi \& Zhang \cite{ETZ-part1} whenever the function $u$ and the nonlinearity $G(.,y,z,\g)$ are $d_\infty-$continuous. By following the same line of argument, the same equivalence of definitions holds under our $d_p-$continuity assumptions.

We next recall the Snell envelop characterization of the optimal stopping problem under nonlinear expectation, see Theorem 3.5 in Ekren, Touzi \& Zhang \cite{ETZ-os}. 

\begin{lem}\label{lem:OSprep}
Let $X:\Th\rightarrow\dbR$ be $d_p$-uniformly continuous. Consider the optimal stopping problem:
\beaa
V(\th) &:=& \sup_{\t\in \cT_{\ch^\th_\d-t}}\ol\cE_L \big[X^\th_\t \big] .
\eeaa
Then, denoting $\hat V_t := V_t 1_{\{t<\ch_\d\}} + V_{\ch_\d-}1_{\{t \ge\ch_{\d}\}}$, we have
\beaa
V_0 = \ol\cE_L [X_{\t^*}]
&\mbox{where}&
\t^* := \inf\{t\ge 0: X_t = \hat V_t \}.
\eeaa
\end{lem}

As a consequence of the last fundamental result from optimal stopping theory, we now provide our main technical substitute for the local compactness argument in the finite-dimensional Crandall-Lions viscosity solutions.

\begin{lem}\label{lem:OS}
Let $u$ be $d_p$-uniformly continuous function satisfying $u({\bf 0})>\ol\cE_L \big[(u - \phi^{\a,\b,\g})(\ch_\d,B) \big]$, for some $\d>0$ and $(\a,\b,\g)\in\dbR\times\dbR^d\times\dbS^d$. 
Then, there exists $\th^*=(t^*,\omega^*)$ such that
 \beaa
 t^*<\ch_\d(\o^*)
 &\mbox{and}&
 (\a,\b + \g\o^*_{t^*},\g) ~\in ~ \underline{\cJ}_L u(\th^*). 
 \eeaa
\end{lem}

\proof Define the optimal stopping problem $V$:
\beaa
V(\th) := \sup_{\t\in\cT_{\ch^\th_\d-t}} \ol\cE_L \big[ X^\th (\t,B) \big].
\eeaa
 with $X := u - \phi^{\a,\b,\g}$. Let $\t^*\in \cT_{\ch_\d}$ be the optimal stopping rule. By Lemma \ref{lem:OSprep} we have
\beaa
\ol \cE_L[X_{\t^*}] = V_0 \ge X_0 > \ol \cE_L [ X_{\ch_\d}] 
&\mbox{and}& 
X_{\t^*} = \hat V_{\t^*},  
\eeaa
So there exists $\o^* \in \O$ such that $t^* := \t^*(\o^*) < \ch_\d(\o^*)$ and $X_{t^*}(\o^*) = \hat V_{t^*}(\o^*) = V_{t^*}(\o^*)$, i.e.
\beaa
u(\th^*) 
& = &  \sup_{\t\in \cT_{\ch^{\th^*}_\d-t^*}} \ol\cE_L \Big[ u^{\th^*}_\t -\a\t -\b\cd B_\t - \frac12(\o^*_{t^*}+B_\t)^{\rm T} \g (\o^*_{t^*}+B_\t) + \frac12 (\o^*_{t^*})^{\rm T} \g \o^*_{t^*}\Big]\\
& = &  \sup_{\t\in \cT_{\ch^{\th^*}_\d-t^*}} \ol\cE_L \Big[ u^{\th^*}_\t -\a\t - (\b+\g \o^*_{t^*}) \cd B_\t - \frac12 B_\t^{\rm T} \g B_\t \Big]
\eeaa
 By the definition of $\ul\cJ_L u$, this means that $(t^*, \o^*)$ is the required point.
\qed

\section{Main result}\label{sec:mainresult}

We shall establish our main comparison result under the following general conditions on the nonlinearity $G$.

\begin{assum}\label{assum:G}
The nonlinearity $G$ satisfies the following conditions:
\\
{\rm (i)} $G$ is elliptic, i.e. $G(\th, y, z, \g) \le G(\th, y,z, \g')$, for $\g\le \g'$.
\\
{\rm (ii)} $G$ is $d_p$-uniformly continuous in $\th$, uniformly in $(y,z,\g)$, i.e. for some continuity modulus $\rho^G$:
\beaa
|G(\th,y,z,\g) - G(\th',y,z,\g) | &\le & \rho^G(d_p (\th,\th')), \q\mbox{for all}\q (y,z,\g)\in \dbR\times\dbR^d\times \dbS^d.
\eeaa
{\rm (iii)} $G$ is uniformly Lipschitz continuous in $(y,z,\g)$, i.e. there is a constant $L_0$ such that
\beaa
|G(\th, y,z,\g)-G(\th, y',z',\g')| &\le & L_0 \big(|y-y'|+|z-z'|+|\g-\g'|\big), \q\mbox{for all}\q \th\in\Th.
\eeaa 
\end{assum}

The first condition restricts the path-dependent PDE to the parabolic case. The remaining technical conditions are required in our proofs. In contrast with the comparison result established in Ekren, Touzi \& Zhang \cite{ETZ-part2}, we emphasize that the above conditions do not exclude degenerate path-dependent second order PDEs. In particular, our main result, Theorem \ref{thm:comparison} below, holds for first order path-dependent PDEs which satisfy the above conditions (ii)-(iii), and thus covers the path-dependent Hamilton-Jacobi PDEs analyzed in Lukoyanov \cite{Lukoyanov}.

\begin{thm}\label{thm:comparison}
Let $u,v:\Th\rightarrow\dbR$ be bounded and $d_p$-uniformly continuous $\cP_L$-viscosity subsolution and $\cP_L$-viscosity supersolution of \eqref{eq: PPDE}, respectively.  Under Assumption \ref{assum:G}, if $u(T,\cd) \le v(T,\cd)$, then $u\le v$ on $\Th$.
\end{thm}

The proof of this result will be provided in Section \ref{sect:comparison} after the preparations of Section \ref{sect:regularization}. We conclude this short section by some remarks which will be recalled in our subsequent proof of Theorem \ref{thm:comparison}.

\begin{rem}\label{rem:monotone} {\rm
If $u$ is a $\cP_L$-viscosity subsolution of \eqref{eq: PPDE}, then function $\tilde u := e^{-Lt} u$ is a $\cP_L$-viscosity subsolution of
\beaa
-\pa_t \tilde u - \tilde G (\th,\tilde u, \pa_\o \tilde u,\pa^2_{\o\o} \tilde u) &=& 0. 
\eeaa
where $\tilde G$ is non-decreasing in $y$. A similar statement holds for $\cP_L$-viscosity supersolutions. 

In view of this result, , we shall assume throughout the paper, without loss of generality, that nonlinearity $G$ is non-decreasing in $y$. Consequently, we may modify Assumption \ref{assum:G} (iii) as:
\ms

\no (iii') {\it There is a constant $L_0$ such that
\be\label{assum:monotone}
G(\th, y,z,\g)-G(\th, y',z',\g') ~\le ~ L_0 \big((y-y')^++|z-z'|+|\g-\g'|\big), \q\mbox{for all}\q \th\in\Th.
\ee
}}
\end{rem}

\begin{rem}\label{rem:odd}{\rm
Since all $d_p$-uniformly continuous function is $d_q$-uniformly continuous, for $q>p$, it is sufficient to prove the theorem for the largest possible values of $p$. For technical reasons, we shall choose $p$ to be odd and $p>1$.
}
\end{rem}

\section{Regularization}\label{sect:regularization}

In this section, we introduce the crucial regularization $u^n$ of the viscosity subsolution $u$. Recall the $\|\cd\|_p$-norm defined in \eqref{eq:deltap}. For $s\ge 0$, a c\`adl\`ag path $\eta$, and any increasing function $\ell$, we define the penalization function:
\beaa
\F(s,\eta,\th,\ell) & := & \|\ell - I\|_{\infty}^{\frac{2}{3p+3}} +  \| \eta_{s\we\ell(\cd)}-  \o_{t\we\cd} \|_{p+1}^{p+1},~~ \th\in \Th,
\eeaa
where $I:[0,T]\rightarrow [0,T]$ is the identity function and $\|\ell - I\|_{\infty} := \sup_{0\le t\le T} \big|\ell(t)-t \big|$.
Denote ${\bf 0} := (0,0)$, and define the regularization:
\be\label{defn2:un}
u^n(s,\eta)
:= \sup_{\th\in \Th \backslash {\bf 0},~\ell\in \cL_{t,s}}\Big\{ u(\th)-n \F (s,\eta,\th,\ell)
\Big\},
\ee
where, for $s>0$, $\cL_{t,s}$ is the set of the increasing functions $\ell:[0,T]\rightarrow\dbR$ such that
 \bea\label{eq:Lts}
 \ell|_{[0,t]}:[0,t]\rightarrow [0,s]
 &\mbox{is an increasing bijection, and}&
 \ell(r) := r -t+s~\mbox{for}~r\in (t,T],
 \eea
For $s=0$, this set is reduced to a signleton $\cL_{t,0}=\{\ell^{t,0}\}$, with $\ell^{t,0}:[0,T]\rightarrow\dbR$ defined as:
\be\label{eq:Lt0}
\ell^{t,0}(r) = (r-t)1_{ (t,T]}(r), \q\mbox{for all}~~r\in [0,T].
\ee
Notice that for $s>0$, any $\ell\in\cL_{t,s}$ is an injection, so that we can naturally define the inverse function $\ell^{-1}$ on the image of $\ell$. For later use, we also note that for all $\ell \in \cL_{t,s}$, we have
\be\label{eq:observest}
\| \eta_{s\we\ell(\cd)}-  \o_{t\we\cd} \|_{p+1}
~=~ \| \eta_{\ell(t\we\cd)}-  \o_{t\we\cd} \|_{p+1}.
\ee
\begin{rem}\label{rem:ell}{\rm
Notice from \eqref{eq:Lts}, \eqref{eq:Lt0} and \eqref{eq:observest} that the values of the function $\ell$ on $(t,T]$ do not have any impact on the value of the penalization $\F(s,\eta,\th,\ell)$. In order to construct a function $\ell\in \cL_{t,s}$, it is sufficient to define $\ell |_{[0,t]}$, i.e. its values on $[0,t]$. The rest is given by \eqref{eq:Lts} or \eqref{eq:Lt0}.
Defining $\ell$ on the whole interval $[0,T]$ instead of only on $[0,t]$ is only useful for notational simplicity. 
}
\end{rem}

\begin{lem}\label{lem:uniformbound}
The sequence $(u^n)_n$ is non-increasing in $n$. Moreover, for bounded $u$, we have $\|u^n\|_\infty\le \|u\|_\infty$ for all $n\ge 0$.
\end{lem}

\begin{proof}
The non-increase of the sequence $(u^n)_n$ is obvious.
The inequality $u^n\le\|u\|_\infty$ follows immediately from the definition of the regularization $u^n$. On the other hand, for $s>0$, we may find $\o^\e \in \O$ such that $ \| \eta_{ s\we\cd} - \o^\e_{s\we \cd} \|_{p+1} \rightarrow 0$ as $\e\rightarrow 0$.  By taking $\th = (s,\o^\e)$ and $\ell = I$ (so that $\ell \in \cL_{s,s}$) in \eqref{defn2:un}, we get 
 \beaa
 u^n(s,\eta) 
 &\ge& 
 \limsup_{\e \rightarrow 0} \big\{ u(s,\o^\e) -n \| \eta_{ s\we\cd} - \o^\e_{s\we \cd} \|_{p+1}^{p+1} \big\} 
 \;\ge\;
 -\|u\|_\infty. 
 \eeaa
In the remaining case $s=0$, we may find $\o^\e \in \O$ such that $ \| \eta_0 - \o^\e_{\e \we \cd} \|_{p+1} \rightarrow 0$ as $\e \rightarrow 0$.  Then, by taking $ \th = ( \e, \o^\e)$ and $\ell |_{[0,\e]} \equiv 0$ (so that $\ell \in \cL_{\e,0}$) in \eqref{defn2:un}, we also conclude that 
\beaa
u^n(0,\eta) ~\ge ~ \limsup_{\e \rightarrow 0} \big\{ u(\e,\o^\e) -n \big(\e^{\frac{2}{3p+3}} + \|\eta_0 - \o^\e_{\e\we\cd}\|_{p+1}^{p+1} \big) \big\} ~\ge ~ -\|u\|_\infty.
\eeaa
\end{proof}

\subsection{Some properties of the regularization}

Our argument relies on using the regularization \eqref{defn2:un} for piecewise constant paths $\eta$ of the following form. Given $0=s_1 < s_2 <\cds <s_i \le T$ and $x_1,x_2,\cds,x_n \in \dbR^d$, denote:
\bea\label{def:lambda i}
 &\pi_i := (s_1,\cds,s_i),~~
  {\rm x}_i := (x_1,\cds,x_i),~
  \l_i := (\pi_i, {\rm x}_{i-1}),
 ~~\mbox{and}~~
  |{\rm x}_i |_p := \big(\sum_{j=1}^{i} |x_j|^p\big)^{\frac{1}{p}} ,&
\eea
and define the corresponding piecewise constant path:
\be\label{def:piececonstpath}
\eta^{\l_i }_s ( x ) ~ =~  \sum_{j=1}^{i-1} x_j 1_{\{s\ge s_j\}} + x 1_{\{s\ge s_i\}},
\ee
i.e. $\eta^{\l_i}( x )$ is a c\`adl\`ag piecewise constant path with $j$-th jump of size $x_j$ at time $s_j$, for $j\le i-1$, and a last jump of size $x$ at time $s_i$.

The following lemma provides an estimate on $1$-optimal $(\hat\th,\hat\ell)$ in the definition of $u^n(s,\eta)$,  in the case that $\eta = \eta^{\l_i}(x )$. For the sake of clarity, we recall the corresponding notion.

\begin{defn}
We say that $\big(\th^{\d},\ell^{\d}\big)$ is $\d$-optimal in the definition of $u^n(s,\eta)$, if
\be\label{defn:epsoptimal}
\ell^\d \in \cL_{t^\d,s}\q\mbox{and}\q 
u^n(s,\eta) - u(\th^\d) + n \F \big(s,\eta,\th^\d,\ell^\d \big)
~< ~ \d.
\ee
\end{defn}

We shall denote by $\1$ the column vector of ones with appropriate dimension, so that with the notations of \eqref{def:lambda i}, the flat tail of the path $\eta^{\l_i}(x_i)$ is given by:
 \beaa
 \eta^{\l_i}_s(x_i)
 \;=\;
 {\rm x}_i \1
 \;=\; 
 x_1+\ldots + x_i,
 ~~s\ge s_i,
 &\mbox{for all}&
 i.
 \eeaa
 
\begin{lem}\label{lem:firstestimate}
Let  $\l_i$ be as in \eqref{def:lambda i}, $x\in \dbR^d$, ${\rm x}_i:=({\rm x_{i-1}},x)$ (we slightly abuse the notation ${\rm x}_i$ so as to simplify the formulas below), and $s \in [ s_i, T]$. Let $u$ be a bounded function. Then, for a $1$-optimal point $(\hat\th,\hat\ell)$ for $u^n\big(s,\eta^{\l_i}( x)\big)$, we have:
 \be\label{estimate: L2}
 |{\rm x}_i\1-\hat\o_{\hat t}| \le Cn^{-{1\over p+1}},
 ~
 \big\|\hat\ell- I\big\|_\infty \le C n^{-\frac{3p+3}{2}},
 ~\mbox{and}~
 \big\| \eta^{\l_i}( x) - \hat\o_{\hat t \we\cd} \big\|_p
 \le 
 C\Big( n^{-\frac{1}{p+1}} + i \,n^{-\frac{3p+3}{2p}}\big| {\rm x}_i \big|_p  \Big),
 \ee
for some constant $C$ depending only on $T$ and $\|u\|_\infty$.
\end{lem}


\begin{proof}
Set $\eta:=\eta^{\l_i}( x)$. By the uniform bound on $u$ and $u^n$ in Lemma \ref{lem:uniformbound}, it follows from \eqref{eq:stopint} and \eqref{defn:epsoptimal} that
 \bea\label{eq: easyestimate}
 &\big\|\hat\ell- I\big\|_\infty^{2\over 3p+3}
 \;\vee\;
 |\eta_s - \hat\o_{\hat t}|^{p+1}
 \;\vee\;
  \int_0^{\hat t}| \eta_{\hat\ell(t)} - \hat\o_t |^{p+1}dt
  ~\le ~
  \frac{C_0}{n}~:=~\frac{1+2|u|_\infty}{n}.&
 \eea
Since $\eta_s={\rm x}_i\1$, this provides the first two required estimates.

By the Minkowski inequality, the H\"older inequality, and \eqref{eq:stopint}, we have
 \bea
 \big\| \eta - \hat\o_{\cd\we \hat t} \big\|_p
 &\le& 
 \big\| \eta - \eta_{\hat\ell(.\wedge \hat t)} \big\|_p
 +\big\| \eta_{\hat\ell(.\wedge \hat t)}- \hat\o_{\cd \we \hat t} \big\|_p
 \nonumber\\
 &\le& 
 \big\| \eta - \eta_{\hat\ell(.\wedge \hat t)} \big\|_p
 +T^{\frac{1}{p(p+1)}}\big\| \eta_{\hat\ell(.\wedge \hat t)}- \hat\o_{\cd\we \hat t} \big\|_{p+1}
 \nonumber\\
 &=&
 \big\| \eta - \eta_{\hat\ell(.\wedge\hat t)} \big\|_p
 +T^{\frac{1}{p(p+1)}} \Big(\int_0^{\hat t} \big| \eta_{\hat\ell(t)} - \hat\o_{t} \big|^{p+1} dt
                                          + (T-\hat t) \big| \eta_s - \hat\o_{\hat t} \big|^{p+1}
                                   \Big)^{\frac{1}{p+1}}
 \nonumber\\
 &\le&
 \big\| \eta - \eta_{\hat\ell(.\wedge\hat t)} \big\|_p
 +
 \Big((1+T)^2\frac{C_0}{n}\Big)^{\frac{1}{p+1}},
 \label{eq: L2minkow}
 \eea
where the last inequality follows from \eqref{eq: easyestimate}. In the special case $s = s_i = 0$, we have $\eta\equiv \eta_0$ and $\hat\ell |_{[0,\hat t]}\equiv 0$, and we therefore get $\big\| \eta - \eta_{\hat\ell(.\wedge\hat t)} \big\|_p=0$. Otherwise, in the case $s > 0$, denoting $x_i:=x$, we have
 \beaa
 \big\| \eta - \eta_{\hat\ell(.\wedge\hat t)} \big\|_p^p
 &=&
 \int_0^T \Big| \sum_{j=1}^{i} x_j \big(1_{\{t\ge s_j \}} -1_{\{\hat\ell(t \we\hat t)\ge  s_j \}}\big)
               \Big|^p dt 
 \\
 &\le &
 i^p \Big(\sum_{j=1}^i  |x_j|^p \int_0^T \big| 1_{\{t \ge s_j \}} -1_{\{\hat\ell(t \we\hat t)\ge s_j \}}\big| dt
      \Big)
 \\
 &\le &
 i^p \Big(\sum_{j=1}^i  |x_j|^p \big| s_j - {\hat\ell}^{-1}(s_j)\big| 
      \Big)
 ~\le~
 i^p |{\rm x}_i|_p^p \|\hat\ell - I \|_\infty 
 ~\le ~ 
 i^p |{\rm x}_i|_p^p \Big(\frac{C_0}{n}\Big)^{\frac{3p+3}{2}}
 \eeaa
by using again \eqref{eq: easyestimate}. The third required estimate is obtained by plugging the last inequality into \eqref{eq: L2minkow}.
\end{proof}

The previous lemma leads to:

\begin{lem}\label{lem:ppt-rgl}
Let $u$ be bounded and $d_p-$uniformly continuous. Then, $\lim_{n\rightarrow\infty} u^{n}({\bf 0}) = u({\bf 0})$.
\end{lem}

\begin{proof}
First, we clearly have
\beaa
u^{n}({\bf 0}) 
~\ge ~  \limsup_{\e\rightarrow 0} \big( u(\e,0) -n \e  \big) 
~=~  u ({\bf 0}).
\eeaa

On the other hand, for $\e<1$, choose an $\e$-optimal $(\hat\th,\hat\ell)$ in the definition of $u^n(0,0) $. It follows from Lemma \ref{lem:firstestimate} that
 \beaa
 d_p\big({\bf 0}, \hat\th\big) 
 ~=~ |\hat t| + \|\hat\o_{\hat t\we\cd}\|_p
 ~\le~ 
 C\Big( n^{-\frac{3p+3}{2}} + n^{-\frac{1}{p+1}} \Big) 
 ~ =: ~ 
 \d_n,
 \eeaa
where $C$ is a constant independent of $\e$, and thus $\d_n$ does not depend on $\e$. Then
 \beaa
 \big| (u^{n}- u)({\bf 0}) \big|  
 &=& 
 (u^{n} - u)({\bf{0}}) 
 ~\le~ 
 \e +  u(\hat\th) - u({\bf{0}}) 
 \\
 &\le& 
 \e+ \sup_{d_p(\th, {\bf{0}}) \le \d_n}\big(u(\th) - u({\bf 0})\big)
 ~\le~ 
 \e+\rho^u(\d_n) ~\longrightarrow~ \rho^u(\d_n),
 \eeaa
by the uniform continuity of $u$. Since $\delta_n\longrightarrow 0$, this shows that $\lim_{n\rightarrow\infty} u^{n}({\bf 0}) = u({\bf 0})$.
\end{proof}

\subsection{A finite-dimensional regularization}

For any $\l_i$ as in \eqref{def:lambda i}, we now introduce the finite-dimensional function
 \bea\label{def:ulambda}
 u^{n,\l_i}(s,x) 
 &:=& 
 u^n\big(s,\eta^{\l_i}(x)\big),
 ~~(s,x) \in [ s_i ,T]\times \dbR^d.
 \eea
In this subsection, we explore the regularity of this function for fixed $\l_i$.

Notice that, for any $s,\l_i,x$, there exists a sequence $(\o^\e)_\e\subset\O$ such that $\|\eta^{\l_i}_{s\we\cd}(x) - \o^\e_{s\we\cd}\|_p\rightarrow 0$. Then, since $u,v,G$ are assumed to be uniformly continuous on $\Th$, these functions have natural extensions for such c\`ad-l\`ag paths. Similar to \eqref{def:ulambda}, we denote
  \beaa
 u^{\l_i}(s,x) 
 &:=& 
 u\big(s,\eta^{\l_i}(x)\big),
 ~~(s,x) \in [ s_i ,T]\times \dbR^d.
 \eeaa
Our first result provides an estimate of the deviation of the penalization at the final time $T$.

\begin{lem}\label{lem:un(T)}
Let $u$ be bounded and $d_p-$uniformly continuous. Then, for  $\l_i$ as in \eqref{def:lambda i}, $x\in \dbR^d$, and ${\rm x}_i:=({\rm x_{i-1}},x)$, we have 
 \beaa
 \rho_n 
 \;:=\; 
 \sup \big\{\big| (u^{n,\l_i} - u^{\l_i} )(T,x)\big| : 
                 i \le n^{1+\frac{1}{5p}},
                 ~|{\rm x}_i |_p \le n^{\frac{1}{2}+\frac{6}{5p}}
         \big\} 
 &\longrightarrow& 0~~\mbox{as}~~ n\to\infty.
 \eeaa
\end{lem}

\begin{proof} 
Set $\eta : = \eta^{\l_i}(x)$, and choose a sequence of $(\o^\e)_\e\subset \O$ such that $\int_0^{T+1} | \eta_{t}- \o^\e_{t} |^{p+1} dt \longrightarrow 0$, as $\e\to 0$. Then, we clearly have
 \beaa
 u^{n,\l_i}(T,x)
 \;:=\;
 u^{n}(T,\eta ) 
 &\ge& 
 \limsup_{\e\rightarrow 0} u(T, \o^\e) 
 \;=\;
 u^{\l_i}(T,x),
 \eeaa
where the last equality follows from the $d_p$-continuity of $u$.

On the other hand, choosing an $\e$-optimal $(\hat\th,\hat\ell)$ in the definition of $u^n(T, \eta)$, it follows from Lemma \ref{lem:firstestimate} that
 \beaa
 d_p\big((T,\eta), \hat\th\big) 
 &=& 
 |T-\hat t| + \| \eta - \hat\o_{\hat t\we\cd}\|_p  
 \\
 &\le& 
 C\Big( n^{-\frac{3p+3}{2}} + n^{-\frac{1}{p+1}} + i |{\rm x}_i|_{p} n^{-\frac{3p+3}{2p}} \Big)
 \\
 &\le& 
 C\Big( n^{-\frac{3p+3}{2}} + n^{-\frac{1}{p+1}} + n^{1+\frac{1}{5p}} n^{\frac{1}{2}+\frac{6}{5p}} n^{-\frac{3p+3}{2p}} \Big)
 \\
 &\le& 
 C\Big( n^{-\frac{3p+3}{2}} + n^{-\frac{1}{p+1}} + n^{-\frac{1}{10p}} \Big)
 ~ =: ~ \d'_n.
 \eeaa
For the second inequality, we used the constraints $i \le n^{1+\frac{1}{5p}}$, and $|{\rm x}_i|_{p} \le n^{\frac{1}{2}+\frac{6}{5p}}$. Then,
 \beaa
 \big(u^{n,\l_i} - u^{\l_i}\big)(T,x) 
 &\le& 
 \e+ u(\hat\th) - u(T,\eta) \\
 &\le& 
 \e+ \sup_{d_p((T,\eta), \th)\le \d'_n}\big( u(\th) - u(T,\eta)\big)\\
 &\le& 
 \e + \rho^u(\d'_n) ~\longrightarrow ~\rho^u(\d'_n), 
 ~~\mbox{as}~~\e\rightarrow 0.
 \eeaa
Hence, we have $0\le \big(u^{n,\l_i} - u^{\l_i}\big)(T,x) \le \rho^u(\d'_n)$, and the required result follows from the fact that $\d'_n\longrightarrow 0$ as $n\to\infty$.
\end{proof}

We next analyze the regularity of the finite-dimensional regularization.

\begin{lem}
The function $u^{n,\l_i}$ is:
\\
{\rm (i)}\quad $\frac{2}{3p+3}-$H\"older continuous in $s\in(s_i,T]$, uniformly in $x\in\dbR^d$,
\\
{\rm (ii)}\quad locally Lipschitz-continuous in $x\in\dbR^d$, uniformly in $s\in(s_i,T]$,
\\
{\rm (ii)}\quad lower semicontinuous at the points $(s_i,x)$, $x\in\dbR^d$, i.e. $u^{n,\l_i} (s_i,x) \le \liminf_{s' \downarrow s_i, x'\to x} u^{n,\l_i}(s',x')$.
\end{lem}

\begin{proof}
{\bf 1.}\q We first consider $s,s' \in (s_i,T]$, and we estimate the value of $u^{n,\l_i} (s,x)-u^{n,\l_i} (s',x')$ for $x , x' \in\dbR^d$. Choose an $\e$-optimal $(\hat\th,\hat\ell)$ in the definition of $u^n \big(s,\eta^{\l_i}(x)\big) = u^{n,\l_i} (s,x)$. Given this $\hat\ell\in\cL_{\hat t,s}$, we define $\hat\ell' \in \cL_{\hat t,s'}$ by its values on $[0,\hat t]$ (see Remark \ref{rem:ell}):
 \beaa
 \hat\ell'(t) 
 &:=&
 \hat\ell(t)\1_{[0,\hat\ell^{-1}(s_i)]}(t)
 + \Big(s_i + \frac{s'-s_i}{s-s_i}(\hat\ell(t) -s_i)\Big)\1_{(\hat\ell^{-1}(s_i), t^*]}(t),
 ~~t\in [0,\hat t].
 \eeaa 
In particular, we observe that
\be\label{eq:ll'}
(\hat\ell')^{-1}(s_i) = \hat\ell^{-1}(s_i),
\ee
and
\bea
 \sup_{t \le \hat t} | \hat\ell_t -\hat\ell'_t | 
 &\le& 
 \sup_{\hat\ell^{-1}(s_i)\le t \le \hat t} \big| s_i + (\hat\ell(t) - s_i)\frac{s'-s_i}{s-s_i} - \hat\ell(t) \big| 
 \nonumber\\
 &\le& 
 \sup_{\hat\ell^{-1}(s_i)\le t \le \hat t} \big(\hat\ell(t) -s_i\big) \frac{|s-s'|}{s-s_i} 
 \;\le\;  
 |s -s' |.
 \label{diff:ll'}
 \eea
Then we have
 \bea\label{diff:u}
 u^{n,\l_i} (s,x)-u^{n,\l_i} (s',x') 
 &\le& 
 \e + n \Big(  \F \big(s', \eta^{\l}(x'),\hat\th,\hat\ell' \big) -  \F \big(s,\eta^{\l_i}(x), \hat\th, \hat\ell\big)  \Big).
 \eea
It follows from \eqref{diff:ll'} that
 \bea\label{diff:lnorm}
 \|\hat\ell' - I \|_\infty^{\frac{2}{3p+3}}  - \|\hat\ell - I\|_\infty^{\frac{2}{3p+3}} 
 &\le& 
 \sup_{t \le\hat t}|\hat\ell_t - \hat\ell'_t|^{\frac{2}{3p+3}}
 \;\le\; |s -s' |^{\frac{2}{3p+3}}.
 \eea
Moreover, using \eqref{eq:ll'}, we directly estimate that: 
 \bea\label{eq:repeat calculus0}
&&
 \| \eta^{\l_i} (x')_{\hat\ell(\hat t\we\cd)} - \hat\o_{\hat t\we\cd} \|_{p+1}^{p+1} 
 - \| \eta^{\l_i}(x)_{\hat\ell(\hat t\we\cd)} - \hat\o_{\hat t\we\cd} \|_{p+1}^{p+1}  
 \nonumber\\
 &\le & 
 \int_{\hat\ell^{-1}(s_i)}^{T+1}|\sum_{j=1}^{i-1}x_j+x' -\o^*_{t\we t^*}|^{p+1} dt 
 - \int_{\hat\ell^{-1}(s_i)}^{T+1}|\sum_{j=1}^{i-1}x_j+x-\o^*_{t\we t^*} |^{p+1} dt 
 \nonumber\\
 &\le &
 (p+1) |x-x'| \int_{0}^{T+1}  \big( |x-x'|+ | \sum_{j=1}^{i-1}x_j+x-\hat\o_{t \we \hat t}  | \big)^p dt.
 \label{eq:repeat calculus}
 \eea
In view of the control on $\|\hat\o_{\hat t\we\cd}\|_p \le  C$ from Lemma \ref{lem:firstestimate}, this provides the statements (i) and (ii) by plugging \eqref{diff:lnorm} and \eqref{eq:repeat calculus0} into \eqref{diff:u}.
\\
\no {\bf 2.}\q We next estimate the difference of $u^{n,\l_i} (s_i,x)-u^{n,\l_i} (s',x')$ for $s'\in (s_i,T]$. We consider two alternative cases.
\\
{\it \underline{Case 1}: $i>1$.} Then, $s_i >0$. Choose an $\e$-optimal $(\hat\th,\hat\ell)$ in the definition of $ u^{n,\l_i} (s_i,x)$. Given this $\hat\ell\in\cL_{\hat t,s_i}$, we define $\hat\ell' \in \cL_{\hat t+\e,s'}$ by:
 \beaa 
 \hat\ell'(t) 
 :=
 \hat\ell(t) \q\mbox{for}~~ t \le \hat\ell^{-1}(s_i) = \hat t,
 &\mbox{and}&
 \hat\ell'~\mbox{linear on}~[\hat t,\hat t+\e].
 \eeaa
In particular, we note that $(\hat\ell')^{-1}(s_i) = \hat\ell^{-1}(s_i)$.  Then we have
\beaa
u^{n,\l_i} (s_i,x)-u^{n,\l_i} (s',x') 
& \le & 
\e + u(\hat\th) - u(\hat t+\e, \hat\o_{\hat t\we\cd})\\
& & +n \Big( \F\big(s',\eta^{\l_i}(x'),\hat t+\e,\hat\o_{\hat t\we\cd}, \ell'\big) 
-\F \big( s_i, \eta^{\l_i}(x), \hat\th, \hat\ell \big)  \Big).
\eeaa
Note that
 \beaa
 \|\hat\ell' - I\|_\infty^{\frac{2}{3p+3}}  - \|\hat\ell - I\|_\infty^{\frac{2}{3p+3}} 
 &\le& 
 \big( \|\hat\ell - I \|_\infty \vee |s' - \hat t-\e|\big)^{\frac{2}{3p+3}} -\|\hat\ell - I\|_\infty^{\frac{2}{3p+3}} 
 \\
 &\le& 
 \Big( \max \big\{ 0, |s' - \hat t-\e| -  \|\hat\ell - I_{\hat t}\|_\infty \big\} \Big)^{\frac{2}{3p+3}} 
 \\
 &\le& 
 \Big( \max \big\{ 0, |s' -\hat t-\e| - | s_i -\hat t | \big\} \Big)^{\frac{2}{3p+3}}
 \\
 &\le& 
 \big( | s_i -s' | +\e \big)^{\frac{2}{3p+3}}.
 \eeaa
Further, by following the line of calculation in \eqref{eq:repeat calculus}, we obtain that for all $|x|,|x'|\le R$ there is a constant $C$ (dependent on $R$, but independent of $n$) such that
 \beaa
 u^{n,\l_i} (s_i, x)-u^{n,\l_i} (s',x') 
 &\le & 
 \e + \rho^u (\e)+ n\big( (|s_i-s'|+\e)^{\frac{2}{3p+3}}  + C|x-x'| \big) 
 \\
 &\longrightarrow & 
 n(|s_i-s'|^{\frac{2}{3p+3}} + C|x-x'|).
 \eeaa
This implies that (iii) holds in the present.
\\
\ms 
{\it \underline{Case 2}: $i=1$.} Then, $s_i =0$ and $\eta^{\l_i}(x)\equiv x$. Choose an $\e$-optimal $(\hat\th,\hat\ell)$ in the definition of $ u^{n,\l_i} (0 ,x)$. Since $\hat\ell\in\cL_{\hat t,0}$, we have $\hat\ell|_{[0,\hat t]}\equiv 0$. Assume $\e < s'$, and define $\ell' \in \cL_{\hat t+\e,s'}$:
 \beaa 
 \hat\ell'(t) 
 &:=&
 \e \frac{t}{\hat t}\1_{[0,\hat t]}(t)
 +\big(\hat t+\e - t + s' \frac{t-\hat t}{\e}\big)\1_{(\hat t,\hat t+\e]}(t),
 ~~t\in[0,\hat t+\e].
 \eeaa
Then we have
 $$
 u^{n,\l_i} (0,x)-u^{n,\l_i} (s',x') 
 \le 
 \e + u(\hat\th) - u(\hat t+\e, \hat\o_{\hat t\we\cd}) 
+ n\Big( \F\big(s',x',\hat t+\e, \hat\o_{\hat t\we\cd},\hat\ell' \big) 
             - \F\big(0,x,\hat\th,\hat\ell \big) \Big).
 $$
Note that
 \beaa
 \|\hat\ell' - I \|_\infty^{\frac{2}{3p+3}}  - \|\hat\ell - I \|_\infty^{\frac{2}{3p+3}} 
 &\le& 
 \big( |\hat t - \e| \vee |s' -\hat t-\e|\big)^{\frac{2}{3p+3}} 
 - |\hat t|^{\frac{2}{3p+3}} 
 \\
 &\le& 
 \Big( \max \big\{ \e, |s' - \hat t-\e| -  |\hat t| \big\} \Big)^{\frac{2}{3p+3}} 
 \\
 &\le& 
 \big( |s'| +\e \big)^{\frac{2}{3p+3}}.
 \eeaa
Then following the same line of calculation as in \eqref{eq:repeat calculus}, we can again verify that (iii) also holds in this case.
\end{proof}


\subsection{Viscosity solution property of the regularized solutions}

In this subsection, we prove a crucial property of the regularization \eqref{defn2:un}. Namely, the induced finite-dimensional function $u^{n,\lambda_i}$ is a Crandall-Lions viscosity subsolution of the corresponding PDE with an appropriate error term. We recall that this does not imply the stronger claim that $u^{n,\lambda_i}$ is a $\cP-$ viscosity subsolution, since the last notion involves a larger set of test functions. This is a major difference between the approach of this paper and the one followed in our previous paper \cite{RTZ} focused on semilinear path-dependent PDEs.

\begin{prop}\label{prop:PPDE2PDE}
Let $\l_i$ be as in \eqref{def:lambda i}, and $u$ be a $\cP_L$-viscosity subsolution of PPDE \eqref{eq: PPDE} on $[0,T)\times\O$. Then $u^{n,\l_i}$ is a Crandall-Lions viscosity subsolution of the PDE:
 \beaa
 -\pa_s u^{n,\l_i} -G\big(s, \eta^{\l_i}(x),u^{n,\l_i} ,D u^{n,\l_i}, D^2 u^{n,\l_i} \big) 
 - \a^{n}(s) -\b^{u,n}(x) 
 &\le& 0,
 \q \mbox{on}\q
 (s_i,T)\times \dbR^d.
 \eeaa
where, for some constant $C>0$, $\alpha^n$ and $\beta^n$ are given by:
\be\label{def:anC}
\a^{n}(s) ~:= ~ C\Big( n |s- s_i| +  (n |s- s_i|)^{\frac{1}{p+1}} + n^{-\frac{1}{2}} \Big),
\ee
\be\label{def:dnC}
\b^{u,n}(x) ~:= ~ (\rho^G + L_0 \rho^u)\Big(C\big(n^{-\frac{1}{p+1}} + i \big| ({\rm x}_{i-1},x) \big|_p n^{-\frac{3p+3}{2p}}\big)\Big).
\ee
\end{prop}


\begin{proof}
Let $(s,x)\in[s_i,T)\times\dbR^d$, and $(\a,\b,\g)\in \ul J u^{n,\l_i} (s,x)$. Then for all $\e>0$, 
\begin{equation}\label{max-jet}
u^{n,\l_i} (s,x) 
=
\max_{t\in [s,s+h], |y-x|\le h}
\Big\{u^{n,\l_i}(t,y) - (\a+\e)(t-s) -\b \cd (y-x) -\frac12 (y-x)^{\rm T}(\g+\e \dbI_d) (y-x) \Big\},
\end{equation}
for some $h\in(0,1)$, where $\dbI_d$ is the $d\times d$-identity matrix. Let $\ch := \ch_\d \le h$ be a stopping time as in \eqref{eq:Heps}, for some $\d<h$ to be chosen later. From \eqref{max-jet}, we deduce that
\bea\label{eq:defnvs}
u^{n,\l_i} (s,x)
&\!\!>& \!\! 
 {\sf E} ~ :=
\ol\cE_L\Big[u^{n,\l_i}(s+\ch,x+ B_\ch) - (\a+2\e) \ch -\b \cd B_\ch - \frac12 B_\ch^{\rm T}(\g+\e \dbI_d) B_\ch \Big].
~~
\eea 
Our objective is to deduce from this inequality an appropriate point in the $\cP-$subjet of $u$. 

\ms

\no {\rm (i)}\quad For $\hat \e:= u^{n,\l_i} (s,x)-  {\sf E} $, let $(\tilde \th , \tilde \ell)$ be a $(1\we\hat \e)$-maximizer in the definition of $u^{n,\l_i}(s,x)$. Then
\bea\label{eq:1optimal}
 {\sf E}  &<& u(\tilde\th)-n\Phi\big(s,\eta^{\lambda_i}(x),\tilde\theta,\tilde\ell\big).
\eea
We recall from Lemma \ref{lem:firstestimate} that, with ${\rm x}_i:=({\rm x}_{i-1},x)$, we have
\bea\label{est:tildeth-eta}
\big|{\rm x}_i\1-\tilde\o_{\tilde t}\big|\le Cn^{-\frac{1}{p+1}}
&\mbox{and}&
d_p\big((s,\eta^{\l_i}(x)), \tilde\th\big) 
\le
 C\big(n^{-\frac{1}{p+1}} + i\,n^{-\frac{3p+3}{2p}} |{\rm x}_i|_p \big).
\eea
On the other hand, it follows from the definition of $u^n$ in \eqref{defn2:un} that 
\bea\label{eq:otherhandI}
u^{n,\l_i}(s+\ch,x+ B_\ch) 
&\ge& 
u^{\tilde \th}(\ch,B) 
-n \Phi\big(s+\ch,\eta^{\lambda_i}(x+\ch),\tilde\theta,\tilde\ell\big).
\eea
Combining \eqref{eq:defnvs}, \eqref{eq:1optimal} and \eqref{eq:otherhandI}, we get
 \bea\label{eq: ELdef}
 u(\tilde \th) 
 &>& 
 \ol\cE_L\big[u^{\tilde \th}(\ch,B) 
                     - (\a+2\e) \ch 
                     -\b \cd B_\ch 
                     - \frac12 B_\ch^{\rm T}(\g+\e \dbI_d) B_\ch 
                     + n \Delta\Phi \Big],
 \eea
where
\beaa
\Delta\Phi
&:=& 
\Phi\big(s,\eta^{\lambda_i}(x),\tilde\theta,\tilde\ell\big)
- 
\Phi\big(s+\ch,\eta^{\lambda_i}(x+\ch),\tilde\theta,\tilde\ell\big)
\\
&=&
\int_0^{T+1} \big| \eta^{\lambda_i}(x)_{s\wedge\tilde\ell(t)}-\tilde\o_{\tilde t\wedge t}\big|^{p+1}dt
-
\int_0^{T+1} \big| \eta^{\lambda_i}(x+B_\ch)_{s\wedge\tilde\ell(t)}-\tilde\o_{\tilde t\wedge t}\big|^{p+1}dt.
\eeaa

\noindent {\rm (ii)}\quad  In this step, we derive an appropriate minorant of $\Delta\F$. Since $\tilde\ell\in\cL_{\tilde t,s}$, this difference reduces to
 \beaa
 \Delta\Phi
 &=&
 \int_{\tilde\ell^{-1}(s_i)}^{T+1} \big| \eta^{\lambda_i}(x)_{s\wedge\tilde\ell(t)}-\tilde\o_{\tilde t\wedge t}\big|^{p+1}dt
-
\int_{\tilde\ell^{-1}(s_i)}^{T+1} \big| \eta^{\lambda_i}(x+B_\ch)_{s\wedge\tilde\ell(t)}-\tilde\o_{\tilde t\wedge t}\big|^{p+1}dt
\\
&=&
 \int_{\tilde\ell^{-1}(s_i)}^{\tilde t} \Big( \big|\bar x-\tilde \o_t \big|^{p+1}
- \big|{\rm x}_i\1+B_\ch-\tilde \o_t \big|^{p+1}\Big) dt 
\\ &&\hspace{10mm}
+ \int_{\tilde t}^{\tilde t+\ch}\Big(\big|{\rm x}_i\1-\tilde \o_{\tilde t} \big|^{p+1} 
 - \big|{\rm x}_i\1+B_\ch -\tilde\o_{\tilde t}-B_{t-\tilde t} \big|^{p+1} \Big) dt.
\eeaa
We next use the obvious identity $|a^{p+1}-b^{p+1}|\le |a-b|\sum_{j=0}^p |a|^j|b|^{p-j}\le(p+1) |a-b|(|a|+|b|)^p$.
This together with \eqref{est:tildeth-eta} allows to control the integrand of the second term:
\beaa
 \Big| \big|{\rm x}_i\1-\tilde \o_{\tilde t}\big|^{p+1} 
                - \big|{\rm x}_i\1+B_\ch -\tilde\o_{\tilde t}-B_{t-\tilde t}\big|^{p+1}\Big|
 &\!\!\le &\!\! 
 (p+1) \big|B_\ch -B_{t-\tilde t}\big| 
            \big(2|{\rm x}_i\1 -\tilde\o_{\tilde t}|+|B_\ch -B_{t-\tilde t}|\big)^p 
 \\
&\!\!\le &\!\! 
(p+1) \big|B_\ch -B_{t-\tilde t}\big| 
            \big(2Cn^{-\frac{1}{p+1}}+|B_\ch -B_{t-\tilde t}|\big)^p.
 \eeaa
Since $0\le t-\tilde t\le \ch=\ch_\delta$, we see that $|B_\ch -B_{t-\tilde t}|\le 2\delta$. We then obtain for sufficiently small $\delta$:
\beaa
 \Big| \big|{\rm x}_i\1-\tilde \o_{\tilde t}\big|^{p+1} 
                - \big|{\rm x}_i\1+B_\ch -\tilde\o_{\tilde t}-B_{t-\tilde t}\big|^{p+1}\Big|
 &\le &
 2\d(p+1) \big(2Cn^{-\frac{1}{p+1}}+2\delta \big)^p
 \;<\; 
 \e\;n^{-1}. 
\eeaa
Therefore,
  \beaa
  \Delta \F 
  &\!\!\ge& \!\!
  -\e\;n^{-1}\ch
  + \int_{\tilde\ell^{-1}(s_i)}^{\tilde t} 
     \big( \big|{\rm x}_i\1-\tilde \o_t \big|^{p+1}
            - \big|{\rm x}_i\1+B_\ch-\tilde \o_t \big|^{p+1}\big) dt 
 \\
 &\!\!\underset{(\star)}{\ge}&\!\!
 -\e\;n^{-1}\ch
  -\!\!\int_{\tilde\ell^{-1}(s_i)}^{\tilde t} \!\!
 \big[(p+1) B_\ch\cdot({\rm x}_i\1-\tilde \o_t) |{\rm x}_i\1-\tilde \o_t|^{p-1}
 \!\!+ C |B_\ch|^2( |B_\ch|^{p-1} +|{\rm x}_i\1-\tilde \o_t |^{p-1} )\big]dt,
 \eeaa
where the last inequality $(\star)$ follows from an easy calculation reported in Step (v) below. Since $|B_\ch|\le\delta\le 1$, this provides
  \bea
  \Delta \F
  &\!\!\ge&\!\!
 -\e\;n^{-1}\ch
  - \int_{\tilde\ell^{-1}(s_i)}^{\tilde t}
             \big[(p+1)B_H\!\cdot\! ({\rm x}_i\1-\tilde \o_t)|{\rm x}_i\1-\tilde \o_t|^{p-1}
                     + C |B_H|^2(1 +|{\rm x}_i\1-\tilde \o_t |^{p-1} )\big]dt
  \nonumber\\
  &\!\!\ge&\!\!
 -\e\;n^{-1}\ch
 -\!\!\int_{\tilde\ell^{-1}(s_i)}^{\tilde t}\!\!
             \big[(p+1)B_H\!\cdot\! ({\rm x}_i\1-\tilde \o_t)|{\rm x}_i\1-\tilde \o_t|^{p-1}
                     \!\!+\!\! C |B_H|^2(2 +|{\rm x}_i\1-\tilde \o_t |^p )\big]dt.~~
 \label{DeltaPhige}
 \eea
 
\noindent {\rm (iii)}\quad We now deduce from the previous step the corresponding minorant of $ \dbE^\dbP[\Delta \F]$, for an arbitrary $\dbP\in\cP_L$. By Lemma \ref{lem:EH}, we deduce from \eqref{DeltaPhige} that 
 \bea\label{estimate:EDelta}
 \dbE^\dbP[\Delta \F]
 &\ge &
 - \dbE^\dbP[\ch]\Big( \e\;n^{-1}+C\int_{\tilde\ell^{-1}(s_i)}^{\tilde t} \big|{\rm x}_i\1-\tilde \o_t \big|^p dt
                                     + C|\tilde t - \tilde\ell^{-1}(s_i)|  
                           \Big).
 \eea
We shall verify in Steps (vi) and (vii) below that the following estimates hold:
\be\label{claim1}
|\tilde t - \tilde\ell^{-1}(s_i)| \le C n^{-\frac{3p+3}{2}} + |s - s_i|,
\ee
\be\label{claim2}
\int_{\tilde\ell^{-1}(s_i)}^{\tilde t} |{\rm x}_i\1-\tilde \o_t|^p dt 
~\le ~ Cn^{-\frac{p}{p+1}}\Big(n^{-\frac{3p+3}{2}} + |s-s_i |\Big)^{\frac{1}{p+1}},
\ee
where $C$ is a positive constant independent of $n$. Then
 \beaa
 n\; \dbE^\dbP [\Delta \F] 
 & \ge & 
 -\dbE^\dbP[\ch] \big(\e +\a^n(s)\big),
 \eeaa
where $\a^{n}(s)$ is as defined in \eqref{def:anC}. 

\ms

\no {\rm (iv)}\quad We are now ready to prove the required result. Plugging the last minorant in \eqref{eq: ELdef}, we see that
 \beaa
 u(\tilde \th) 
 &>& 
 \ol\cE_L\Big[u^{\tilde \th}(\ch,B) 
                     - (\a + \a^{n}(s) + 3\e) \ch 
                     -\b \cd B_\ch 
                     - \frac12 B_\ch^{\rm T}(\g+\e \dbI_d) B_\ch 
               \Big].
\eeaa
By Lemma \ref{lem:OS}, there is a point $\th^*$ such that
 \bea\label{optimalpoint}
 t^* < \ch (\o^*) 
 &\mbox{and}&
 \big(\a + \a^{n}(s) + 3\e,~
        \b+(\g+\e\dbI_d)\o^*_{t^*}, ~
        \g +\e\dbI_d 
 \big) \;\in \ul\cJ_L u(\tilde t+t^*, \tilde\o\otimes_{\tilde t} \o^*).~~~
 \eea
By choosing $\d$ small enough, we have 
\be\label{est:xiao}
|(\g+\e\dbI_d)\o^*_{t^*}| ~\le ~ |\g+\e\dbI_d | \d ~\le ~ \frac{\e}{L_0}, 
\ee
\be\label{est:tildeth-thstar}
d_p \big(\tilde\th, (\tilde t+t^*, \tilde\o\otimes_{\tilde t} \o^*)\big)
=
t^* + \Big(\int_{\tilde t}^{T+1} |\o^*_{(t-\tilde t)\we t^*}|^p dt  \Big)^{\frac{1}{p}}~\le ~ C\d ~\le ~ \e.
\ee
Since $u$ is a $\cP_L$-viscosity sub-solution, it follows from \eqref{optimalpoint} that
\beaa
-\a - \a^{n}(s) - 3\e -G \big(\tilde t+t^*,  \tilde\o\otimes_{\tilde t} \o^*, u ,\b+ (\g+\e\dbI_d)\o^*_{t^*}, \g+\e\dbI_d \big) \le 0.
\eeaa
Recall that $u$ is $d_p-$uniformly continuous. By using Assumption \ref{assum:G} and the estimates \eqref{est:tildeth-eta}, \eqref{est:xiao} and \eqref{est:tildeth-thstar}, we deduce from the last inequality that
 \beaa
 -\a - \a^{n}(s) - 4\e - \b^{u,n}(x) -G(s, \eta^{\l_i}(x), u\big(s, \eta^{\l_i}(x)\big),\b, \g+\e\dbI_d) 
 &\le& 
 0,
 \eeaa  
where $\b^{u,n}(x)$ is as defined in \eqref{def:dnC}. Finally, sending $\e\rightarrow 0$ and using the monotonicity assumption in Remark \ref{rem:monotone}, we obtain
\beaa
-\a  - \a^{n}(s) -\b^{u,n}(x) -G\big(s,\eta^{\l_i}(x), u^{n,\l_i}(x), \b,\g \big) &\le & 0.
\eeaa

\ms

\no {\rm (v)} {\it Proof of $(\star)$}\q Clearly, this inequality is implied by
\beaa
|a+b|^{p+1} \le |a|^{p+1} + (p+1)a\cdot b |a|^{p-1} +C|b|^2(|b|^{p-1}+|a|^{p-1}),\q\mbox{for}~~a,b\in \dbR^d,\q\mbox{for some}~~C\ge 0,
\eeaa
which we now verify. Since $p>1$ is odd, see Remark \ref{rem:odd}, we have
 \beaa
 |a+b|^{p+1} \;=\; 
 |a|^{p+1} + (p+1) a\cdot b |a|^{p-1} + R,
 &\mbox{where}&
 R := \!\!\!\!\sum_{k+j\le \frac{p+1}{2},k+2j\ge 2} 
                (a\cdot b)^k |b|^{2j} |a|^{p+1-2k-2j}.
 \eeaa
The required inequality follows from the existence of a constant $C$, depending only on $p$, such that
 \beaa
 |R|
 ~\le ~ C\sum_{k=2}^{p+1} |b|^k |a|^{p+1-k}
 ~\le ~ C|b|^2 \sum_{k=0}^{p-1} |b|^k |a|^{p-1-k} 
 ~\le ~ C p |b|^2 \big(|b|^{p-1} + |a|^{p-1}\big).
 \eeaa

\no (vi) \q {\it Proof of \eqref{claim1}.}\quad Recall the estimates in \eqref{estimate: L2}. Since $\tilde \ell \in \cL_{\tilde t, s}$, we have
\beaa
|\tilde t -\tilde \ell^{-1}(s_i) |  & = &  |\tilde \ell^{-1}(s) - \tilde \ell^{-1}(s_i)|\\
&\le & |\tilde \ell^{-1}(s) - s | + |s-s_i| +|\tilde \ell^{-1}(s_i)-s_i|\\
&\le & C n^{-\frac{3p+3}{2}}+|s-s_i|.
\eeaa
For the last inequality, we used the fact that $(\tilde\th,\tilde \ell)$ is $1$-optimal in the definition of $u^{n,\l_i}(s,x)$.

\ms

\no (vii) \q {\it Proof of \eqref{claim2}.}\quad We directly estimate that
 \beaa
 \int_{\tilde \ell^{-1}(s_i)}^{\tilde t}|{\rm x}_i\1-\tilde \o_t|^p dt
 &\le & 
  \Big(\int_{\tilde \ell^{-1}(s_i)}^{\tilde t} |{\rm x}_i\1-\tilde \o_t |^{p+1} dt
  \Big)^{\frac{p}{p+1}}
   \Big(\tilde t-\tilde \ell^{-1}(s_i)
   \Big)^{\frac{1}{p+1}}\\
&\le & 
Cn^{-\frac{p}{p+1}}\big(n^{-\frac{3p+3}{2}} + |s-s_i |\big)^{\frac{1}{p+1}},
\eeaa
where the last inequality follows from \eqref{claim1} together with the $1$-optimality of $(\tilde \th,\tilde\ell)$.
\end{proof}

\section{Comparison result}\label{sect:comparison}

In this section, we fix $a$, $m:=m_n\in\dbN$, and the partition $(s^n_i)_i$ as follows: 
 \beaa
 0<a<(5p)^{-1},~~
 m_n:= \lfloor n^{1+a}+1\rfloor,
 &\mbox{and}& 
 s^n_i :=(i-1)m_n^{-1}T, ~~i=1,\ldots,m_n+1,
 \eeaa
where $\lfloor \alpha\rfloor$ denotes the largest integer minorant of $\alpha$. We fix a piecewise constant path with jumps occurring at $\{s^n_j\}_{j\le i}$, for all $i\le m_n$:
 \beaa
 \eta^{\l^n_i}(x)
 &\mbox{with}&
 \l^n_i := (\pi^n_i, {\rm x}^n_{i-1}),
 ~\pi^n_i := (s^n_1,\cds,s^n_i),~\mbox{and}~
 {\rm x}^n_{i-1}:=(x^n_1,\cds,x^n_{i-1})\in\dbR^{i-1}.
 \eeaa
The following is a direct corollary of Proposition \ref{prop:PPDE2PDE}.

\begin{cor}\label{cor:PDEsub}
Function $u^{n,\l^n_i}$ is a Crandall-Lions viscosity subsolution of the PDE:
 \beaa
 -\pa_s u^{n,\l^n_i} -G\big(s,\eta^{\l^n_i}(x),u^{n,\l^n_i}, D u^{n,\l^n_i}, D^2 u^{n,\l^n_i} \big) 
 - R^{u,n}(x) 
 &\le& 
 0,
 ~~\mbox{on}~~
 \big(s^n_i, s^n_{i+1}\big)\times \dbR^d,
\eeaa
where $R^{u,n}(x) :=  Cn^{-\frac{a}{p+1}} +\b^{u,n}(x)$. Moreover
\\
{\rm (i)}\quad $u^{n,\l^n_i}(s^n_{i+1},x^n_{i+1}) = u^{n,\l^n_{i+1}}(s^n_{i+1},0)$, 
\\
{\rm (ii)}\quad $u^{n,\l^n_i}$ is locally $\frac{2}{3p+3}-$H\"{o}lder-continuous in $s$, Lipschitz-continuous in $x$ on $\big(s^n_i,s^n_{i+1}\big)$,
\\
{\rm (iii)}\quad $u^{n,\l^n_i}$ is lower-semicontinuous at $s^n_i$, i.e. $\liminf_{s' \downarrow s^n_i, x'\rightarrow x} u^{n,\l^n_i}(s',x')\ge u^{n,\l^n_i} (s^n_i,x)$.
\end{cor}

We next state the similar result for supersolutions.
Let $v$ be a bounded and uniformly continuous $\cP_L$-viscosity supersolution. Then we introduce the regularization:
 \beaa
 v^n(s,\eta)
 &:=& 
 \inf_{\th \in \Th\backslash {\bf 0},~\ell\in \cL_{t,s}}\big\{ v(\th) + n\F (s,\eta,\th,\ell) \big\}.
 \eeaa
By the same arguments as in the previous section, we have that the function $v^{n,\l^n_i}$ satisfies the corresponding symmetric properties.

\begin{cor}\label{cor:PDEsuper}
Function $v^{n,\l^n_i}$ is a Crandall-Lions viscosity supersolution of the PDE:
 \beaa
 -\pa_s v^{n,\l^n_i} -G\big(s,\eta^{\l^n_i}(x),v^{n,\l^n_i}, Dv^{n,\l^n_i}, D^2v^{n,\l^n_i} \big) 
 + R^{v,n}(x) 
 &\ge& 
 0,
 ~~\mbox{on}~~
 \big(s^n_i, s^n_{i+1}\big)\times \dbR^d,
\eeaa
where $R^{v,n}(x) :=  Cn^{-\frac{a}{p+1}} +\b^{v,n}(x)$. Moreover
\\
{\rm (i)}\quad $v^{n,\l^n_i}(s^n_{i+1},x^n_{i+1}) = v^{n,\l^n_{i+1}}(s^n_{i+1},0)$, 
\\
{\rm (ii)}\quad $v^{n,\l^n_i}$ is locally $\frac{2}{3p+3}-$H\"{o}lder-continuous in $s$, Lipschitz-continuous in $x$ on $\big(s^n_i,s^n_{i+1}\big)$,
\\
{\rm (iii)}\quad $v^{n,\l^n_i}$ is upper-semicontinuous at $s^n_i$, i.e. $\limsup_{s' \downarrow s^n_i, x'\rightarrow x} v^{n,\l^n_i}(s',x')\le v^{n,\l^n_i} (s^n_i,x)$.
\end{cor}

As a final ingredient for our proof of the comparison result, we introduce for $(s,x)\in[s^n_i,T]\times\dbR^d$:
 \beaa
 u^{n,\l^n_i,\k}(s,x) 
 &: =& 
 e^{2L_0 s} u^{n,\l^n_i}(s,x) - \frac{\k n^{-1-a}}{s- s^n_i} - \frac{1}{2n} |x|^2,
 \\
 v^{n,\l^n_i,\k}(s,x) 
 &: =& 
 e^{2L_0 s} v^{n,\l^n_i}(s,x) + \frac{\k n^{-1-a}}{s- s^n_i} + \frac{1}{2n} |x|^2.
 \eeaa
By the standard change of variable in the Crandall-Lions theory of viscosity solutions, we deduce from Corollaries \ref{cor:PDEsub} and \ref{cor:PDEsuper} that the functions $u^{n,\l^n_i, \k}$ and $v^{n,\l^n_i, \k}$ are respectively viscosity sub-solution and super-solution on $(s^n_i,s^n_{i+1}]\times \dbR^d$ of
 \beaa
 -\pa_s u^{n,\l^n_i,\k}  
 - \bar G\big(s,\eta^{\l^n_i}(x), u^{n,\l^n_i,\k},D u^{n,\l^n_i,\k}+ \frac{x}{n}, D^2 u^{n,\l^n_i,\k}+\frac{1}{n} \dbI_d \big) 
 - R^{u,n}(x)
 &\le & 0,
 \\
 -\pa_s v^{n,\l^n_i,\k} 
 - \bar G\big(s,\eta^{\l^n_i}(x), v^{n,\l^n_i,\k},D v^{n,\l^n_i,\k}- \frac{x}{n}, D^2 v^{n,\l^n_i,\k}-\frac{1}{n}\dbI_d \big) 
 + R^{v,n}(x)
 &\ge & 
 0,
\eeaa
where $\bar G(\th, y, z,\g) = -2L_0 y + e^{2L_0 t}G(\th, e^{-2L_0 t}y, e^{-2L_0 t}z, e^{-2L_0 t}\g)$. In particular, note that
\beaa
\bar G(\th,y,z,\g) - \bar G(\th, y',z,\g) ~ \le ~ -L_0 (y-y')^+ +3L_0 (y-y')^- ,
\eeaa
and therefore
 \bea\label{barG}
 L_0(y-y') 
 &\le& 
 \big(\bar G(\th,y',z,\g) - \bar G(\th, y,z,\g)\big)^+
 ~\le~ 
 \big|\bar G(\th,y',z,\g) - \bar G(\th, y,z,\g)\big|.
 \eea

\vspace{3mm}

\no{\bf Proof of Theorem \ref{thm:comparison}}\q Without loss of generality, we only prove $(u-v)({\bf 0}) \le 0$.
\\
{\bf 1.}\quad Following the classical argument of doubling variables, for fixed $n$ and $i$, we define 
\beaa
w^{\k,\e}(s,x,x') & := & u^{n,\l^n_i,\k}(s,x) -v^{n,\l^n_i,\k}(s,x') - \frac{1}{2 \e} |x-x'|^2.
\eeaa
There is a constant $C>0$ only dependent on $T$ and the bound of $u,v$, and a point $(\hat s_{\k,\e}, \hat x_{\k,\e},\hat x'_{\k,\e})\in Q_n:=[s^n_i +n^{-1-a}\k/C ,s^n_{i+1}]\times O_{C\sqrt{n}} \times O_{C\sqrt{n}}$ such that
 \beaa
 w^{\k,\e}(\hat s_{\k,\e},\hat x_{\k,\e},\hat x'_{\k,\e}) 
 &=& 
 \max_{(s,x,x')\in (s^n_i,s^n_{i+1}]\times \dbR^d \times \dbR^d}
 w^{\k,\e}(s,x,x').
 \eeaa
Since $Q_n$ is compact, $\big\{(\hat s_{\k,\e}, \hat x_{\k,\e},\hat x'_{\k,\e})\big\}_\e$ has a converging sub-sequence whose limit is denoted by $(\hat s_{\k}, \hat x_{\k},\hat x'_{\k})$. In particular, it is easy to show that $\hat x_\k = \hat x'_\k$. \\
{\bf 2.}\quad We continue by discussing two alternative cases.
\\
\no{\it Case 1.}\q Suppose that there are only a finite number of $\k$ such that  $\hat s_\k < s^n_{i+1}$, and thus there is sub-sequence still denoted as $\{\hat s_\k\}_\k$ such that $\hat s_\k \equiv s^n_{i+1}$. By Corollaries \ref{cor:PDEsub} (ii) and \ref{cor:PDEsuper} (ii), $u^{n,\l^n_i,\k}$ and $v^{n,\l^n_i,\k}$ are continuous on $(s^n_i, s^n_{i+1}]$. This provides for all $s \in (s^n_i, s^n_{i+1}]$ that
 \beaa
 \big(u^{n,\l^n_i,\k} -v^{n,\l^n_i,\k}\big)(s,0) 
 &\le & 
 \lim_{\e\to 0} \big\{ u^{n,\l^n_i,\k}(\hat s_{\k,\e},\hat x_{\k,\e}) -v^{n,\l^n_i,\k}(\hat s_{\k,\e},\hat x'_{\k,\e})\big\} 
 \\
 &\le & 
 \big( u^{n,\l^n_i,\k} -v^{n,\l^n_i,\k} \big) (s^n_{i+1},\hat x_\k) 
 \\
 &\le & 
 \sup_{|x|\le C\sqrt{n}} \big(u^{n,\l^n_i,\k} -v^{n,\l^n_i,\k}\big) (s^n_{i+1},x)
 \\
 &\le & 
 e^{2L_0 s^n_{i+1}}\sup_{|x|\le C\sqrt{n}} \big(u^{n,\l^n_i} -v^{n,\l^n_i}\big)(s^n_{i+1},x)
 \eeaa
We next send $\k\searrow 0$ and then $s \searrow s^n_i$. By the semicontinuity properties of $u^{n,\l^n_i}$ and $v^{n,\l^n_i}$ at $s^n_i$ stated in Corollaries \ref{cor:PDEsub} (iii) and \ref{cor:PDEsuper} (iii), we obtain
 \beaa
 (u^{n,\l^n_i}-v^{n,\l^n_i})(s^n_i,0) 
 &\le& 
 e^{\frac{2L_0 T}{n^{1+a}}}\sup_{|x|\le C\sqrt{n}}\big(u^{n,\l^n_i} -v^{n,\l^n_i}\big)(s^n_{i+1},x).
 \eeaa
\no {\it Case 2.}\q Otherwise, there is a sub-sequence still denoted by $\{\hat s_\k\}_\k$ such that $\hat s_\k < s^n_{i+1}$ for each $\k$. Then, by the Crandall-Ishii Lemma in the parabolic case (see for example Theorem 12.2 on page 38 in \cite{Crandall}), there are $\a,X,Y$ such that
 \beaa
 &
 \big(\a,\e^{-1}(\hat x_{\k,\e} -\hat x'_{\k,\e}) ,X \big)\in \ul J u^{n,\l^n_i, \k}(\hat s_{\k,\e}, \hat x_{\k,\e}),
 \q 
 \big(\a,\e^{-1}(\hat x_{\k,\e} -\hat x'_{\k,\e}),Y \big)\in \ol J v^{n,\l^n_i, \k}(\hat s_{\k,\e}, \hat x'_{\k,\e}),
 &
 \eeaa
and $ X\le Y$. By the viscosity properties of $u^{n,\l^n_i, \k}$ and $v^{n,\l^n_i, \k}$ of Corollaries \ref{cor:PDEsub} a,d \ref{cor:PDEsuper}, respectively, this implies that
 \beaa
 &- R^{u,n}(\hat x_{\k,\e})
   - \bar G\Big(\hat s_{\k,\e}, 
                       \eta^{\l^n_i}( \hat x_{\k,\e}),
                       u^{n,\l^n_i,\k} ,
                       \e^{-1}(\hat x_{\k,\e} -\hat x'_{\k,\e})+\frac{\hat x^{\k,\e}}{n}, 
                       X+\frac{1}{n}\dbI_d\Big) 
  &
  \\
  & \q\q\q\q\q\q   \le ~ 0~\le ~   
      R^{v,n}(\hat x'_{\k,\e})
      - \bar G\Big(\hat s_{\k,\e}, 
                          \eta^{\l^n_i}(\hat x'_{\k,\e}), 
                          v^{n,\l^n_i,\k} , 
                          \e^{-1}(\hat x_{\k,\e} -\hat x'_{\k,\e}) -\frac{\hat x'_{\k,\e}}{n}, 
                          Y - \frac{1}{n}\dbI_d \Big). &
 \eeaa
By \eqref{barG}, we obtain
 \beaa
 & & 
 L_0\big(u^{n,\l_i,\k}(\hat s_{\k,\e}, \hat x_{\k,\e}) - v^{n,\l_i,\k} (\hat s_{\k,\e}, \hat x'_{\k,\e}) \big)
 \\
 & \le &  
 L_0\big(2n^{-1} + n^{-1} | \hat x_{\k,\e}+\hat x'_{\k,\e}| \big) 
 +\rho^G\big(C|\hat x_{\k,\e} - \hat x'_{\k,\e}|\big)
 +R^{u,n}(\hat x_{\k,\e}) 
 +R^{v,n}(\hat x'_{\k,\e})
 \\
 & \le & 
 2 Cn^{-\frac{a}{p+1}}+ 2L_0\big(n^{-1} + C n^{-\frac12} \big)
 +\rho^G\big(C|\hat x_{\k,\e} - \hat x'_{\k,\e}|\big) 
 + \ol\rho \Big(n^{-\frac{1}{p+1}} + i \,n^{-\frac{3p+3}{2p}}\big|{\rm x}^n_i\big|_p \Big),
 \eeaa
where $\ol \rho (\cd) := (2\rho^G +L_0\rho^u +L_0\rho^v)(C\cd)$ and ${\rm x}^n_i:=({\rm x}^n_{i-1},x)$.
Hence for any $s \in (s^n_i, s^n_{i+1}]$ we have
 \beaa
 \big( u^{n,\l^n_i,\k}- v^{n,\l^n_i,\k}\big) (s, 0)
 & \le & 
 \limsup_{\e\to 0}
 \big(u^{n,\l_i,\k}(\hat s_{\k,\e}, \hat x_{\k,\e}) 
        - v^{n,\l_i,\k} (\hat s_{\k,\e}, \hat x'_{\k,\e})\big)
 \\
 & \le &  
 Cn^{-\frac{a}{p+1}} + \ol\rho \Big(n^{-\frac{1}{p+1}} + i\, n^{-\frac{3p+3}{2p}}\big|{\rm x}^n_i\big|_p \Big)
 \\
 & \le & 
 Cn^{-\frac{a}{p+1}} + \ol\rho \Big(n^{-\frac{1}{p+1}} + i\, n^{-\frac{3p+3}{2p}}\big(|{\rm x}^n_{i-1}|_p +C\sqrt n \big) \Big).
 \eeaa
We next let $\k\searrow 0$ and then let $s \searrow s^n_i$. By the semicontinuity properties of $u^{n,\l^n_i}$ and $v^{n,\l^n_i}$ stated in Corollaries \ref{cor:PDEsub} (iii) and \ref{cor:PDEsuper} (iii), we obtain
 \beaa
 \big( u^{n,\l^n_i}- v^{n,\l^n_i}\big) (s^n_i,0) 
 & \le & 
 Cn^{-\frac{a}{p+1}}+ \ol\rho \Big(n^{-\frac{1}{p+1}} + i\,n^{-\frac{3p+3}{2p}} \big(|{\rm x}^n_{i-1}|_p +C\sqrt n \big)\Big).
 \eeaa
{\bf 3.}\quad By the results of Cases 1 and 2 in the previous Step 2, we conclude that
\beaa
(u^{n,\l^n_i}-v^{n,\l^n_i})(s^n_i,0) &\le & \max \Big\{
 e^{\frac{2L_0T}{n^{1+a}}}\sup_{|x| \le C\sqrt{n}}\big(u^{n,\l^n_i} -v^{n,\l^n_i}\big)(s^n_{i+1},x),\\
 && \q\q\q Cn^{-\frac{a}{p+1}}+ \ol\rho \Big(n^{-\frac{1}{p+1}} + i \big(|{\rm x}^n_{i-1}|_p +C\sqrt n \big)  n^{-\frac{3p+3}{2p}}\Big) 
 \Big\}.
\eeaa
We next use Corollaries \ref{cor:PDEsub} (ii) and \ref{cor:PDEsuper} (ii) so that by direct iteration, it follows that:
\beaa
(u^{n}-v^{n})({\bf 0}) 
&\le & 
e^{2L_0T} 
\max \Big\{ \sup_{i \le n^{1+a},|{\rm x}_i |_p \le Cn^{\hat a}} \big( u^{n} - v^n\big) \big(T,\eta^{\l_i}(x_i)\big) , \\
&& \q\q\q\q\q\q\q  Cn^{-\frac{a}{p+1}} +  \ol\rho \Big( n^{-\frac{1}{p+1}} + C n^{1+a} n^{\hat a} n^{-\frac{3p+3}{2p}}\Big)\Big\},\\
&\le & e^{2L_0T} \max \Big\{ \sup_{i \le n^{1+a},|{\rm x}_i |_p \le Cn^{\hat a}} \big( u^{n} - v^n\big) \big(T,\eta^{\l_i}(x_i)\big), \\
&& \q\q\q\q\q\q\q\q\q Cn^{-\frac{a}{p+1}} +  \ol\rho \Big( n^{-\frac{1}{p+1}} + C n^{-\frac{1}{2p} +a +\frac{a}{p}} \Big)\Big\},
\eeaa
where $\hat a := \frac{1+a + p/2}{p} =  \frac12 +\frac{a+1}{p} \le \frac12 + \frac{6}{5p}$. Since $u(T,\cd) \le v(T,\cd)$, we have
\beaa
\big( u^{n} - v^n \big)(T,\eta)
&\le & \big( u^{n}- u\big)(T,\eta) + \big( u- v\big)(T,\eta) -\big( v^{n}- v\big)(T,\eta)\\
&\le & \big| \big(u^{n}- u\big)(T,\eta) \big| + \big| \big(v^{n} - v\big) (T,\eta) \big|
\eeaa
Recall the notation $\rho_n$ introduced in Lemma \ref{lem:ppt-rgl}. Since $a<\frac{1}{5p}$, we have
\beaa
(u^{n}-v^{n})({\bf 0})  ~\le  ~ e^{2L_0 T} \max\Big\{2 \rho_n, Cn^{-\frac{a}{p+1}} +  \ol\rho \big(n^{-\frac{1}{p+1}} +C n^{-\frac{1}{10p}}\big)\Big\}.
\eeaa
By Lemma \ref{lem:ppt-rgl}, the regularizations $u^n$ and $v^n$ converge to $u$ and $v$, respectively. Then, by sending $n\to\infty$, we obtain the required result
 \beaa
 (u-v)({\bf 0}) 
 &\le& 
 0.
 \eeaa
\qed

\section{Appendix}\label{sec:appendix}

In this section we provide sufficient conditions for the value function of a stochastic control problem to be $d_p-$uniformly continuous.


\begin{eg}\label{eg:L1continuous}
Let $X$ be a controlled diffusion $dX^\a_s = \si (s,X^\a,\a_s) dW_s,$
where $W$ is a Brownian motion, where the function $\si: (\th,\a)\mapsto \si(\th,\a)$ is bounded and $d_p-$Lipschitz continuous in $\th$. Denote the shifted process:
 \beaa
 dX^{\a,\th}_s 
 &=& 
 \si^{\th} (s,X^{\a,\th},\a_s) dW_s.
 \eeaa
We consider the stochastic control problem:
 \beaa
 u_0 
 &:=& 
 \sup_{\|\a\|_\infty \le 1} \dbE\big[g(X^{\a}_{T\we\cd})\big],
 \eeaa
where the function $g$ is uniformly continuous in $L^p$-norm, i.e.  $|g(\o)-g(\o')|\le \rho\big( \| \o-\o' \|_p \big)$ (without loss of generality we may assume that $\rho$ is concave). Introduce the dynamic version:
 \beaa
 u(\th) 
 &:=& 
 \sup_{\|\a\|_\infty \le 1} \dbE\Big[g^{\th}\big(X^{\a,\th}_{(T-t)\we\cd}\big)\Big].
 \eeaa
Our main objective in this section is to prove that
 \be\label{ud1unifcont}
 \mbox{the function $u$ is $d_p-$uniformly continuous}. 
 \ee
To see this, we first estimate that
 \bea
 |u(t,\o) - u(t,\o')|  
 &\le& 
 \sup_\a \dbE \big| g^{t,\o}(X^{\a,t,\o}_{(T-t)\we\cd})-g^{t,\o'}(X^{\a,t,\o'}_{(T-t)\we\cd})\big|
 \nonumber\\
 &\le& 
 \sup_\a \dbE\Big[\rho\Big(\|\o_{t\we\cd}-\o'_{t\we\cd}\|_p 
                             + \|X^{\a,t,\o}_{(T-t)\we\cd}-X^{\a,t,\o'}_{(T-t)\we\cd}\|_p \Big)\Big]
 \nonumber\\
 &\le& 
 \sup_\a \rho\Big( d_p\big((t,\o),(t,\o')\big) + \dbE \|X^{\a,t,\o}_{(T-t)\we\cd} - X^{\a,t,\o'}_{(T-t)\we\cd}\|_p \Big),
 \label{step1-d1cont}
 \eea
where we applied Jensen's inequality in the last step. We next focus on the estimation of
 \beaa
 && \dbE \|X^{\a,t,\o}_{s\we (T-t)\we\cd}-X^{\a,t,\o'}_{s\we(T-t)\we\cd}\|_p^{2p}\\
 &\le& 
 C \int_0^{T+1}  \dbE\Big| \int_0^{s\we (T-t)\we r}\big(\si^{t,\o} (\l,X^{\a,t,\o},\a_\l)
                                                                   - \si^{t,\o'} (\l,X^{\a,t,\o'},\a_\l) 
                                                            \big)dW_\l \Big|^{2p} dr 
 \\
 &\le& 
 C \int_0^{T+1}  \dbE \int_0^{s\we (T-t) \we r}\big| \si^{t,\o} (\l,X^{\a,t,\o},\a_\l)
                                                             - \si^{t,\o'} (\l,X^{\a,t,\o'},\a_\l) \big|^{2p} d\l  dr 
 \\
 &\le& 
 (2C_{\rm lip})^{2p}C(T+1) \int_0^s \Big(d_p\big((t,\o),(t,\o')\big)^{2p} + \dbE \|X^{\a,t,\o}_{\l\we (T-t)\we\cd}-X^{\a,t,\o'}_{\l\we (T-t)\we \cd} \|_p^{2p} \Big) d\l,
 \eeaa
 where $C_{\rm lip}$ is the Lipschitz constant of $\si$.
By the Gronwall inequality, this provides
 \beaa
 \dbE \|X^{\a,t,\o}_{(T-t)\we\cd}-X^{\a,t,\o'}_{(T-t)\we\cd}\|_p^{2p} 
 &\le& 
 \tilde C d_p\big((t,\o),(t,\o')\big)^{2p}, \q\mbox{with}~~  \tilde C  = (2C_{\rm lip})^{2p}C(T+1) e^{(2C_{\rm lip})^{2p}C(T+1)T}.
 \eeaa
Plugging the last inequality into \eqref{step1-d1cont}, we get
 \bea\label{fix_t}
 |u(t,\o) - u(t,\o')| 
 &\le& 
 \rho\Big(\big(1+\tilde C^{\frac{1}{2p}}\big) d_p\big((t,\o),(t,\o')\big)\Big).
 \eea

We next estimate $|u(t,\o) - u(t',\o)|$ for $t < t'$. By the dynamic programming, we have
 \beaa
 |u(t,\o) - u(t',\o) | 
 &=& 
 \big|\sup_\a \dbE[u^{t,\o}(t'-t, X^{\a,t,\o})] -u(t',\o) \big| 
 \\
 &\le&  
 \sup_\a \dbE \big| u^{t,\o}(t'-t, X^{\a,t,\o})-  u(t',\o) \big|
 \\
 &\le & 
 \sup_\a \dbE \Big[ \rho\Big(\big(1+\tilde C^{\frac{1}{2p}}\big) d_p\big((t',\o\otimes_{t} X^{\a,t,\o}),(t',\o)\big)\Big)\Big] 
 \\
 &\le & 
 \sup_\a \rho\Big( \big(1+\tilde C^{\frac{1}{2p}}\big) \dbE \Big[ d_p\big((t',\o\otimes_{t} X^{\a,t,\o}),(t',\o)\big)\Big] \Big)\\
 &\le &
  \sup_\a \rho\Big( \big(1+\tilde C^{\frac{1}{2p}}\big) \Big(\int_t^{T+1} \dbE \big|X^{\a,t,\o}_{(s-t)\we (t'-t)} \big|^p ds \Big)^{\frac{1}{p}}\Big] \Big),
 \eeaa
where we applied the result of \eqref{fix_t} in the third inequality. Finally using the classical estimate
\beaa
\sup_{\a,t,\o,r}\dbE \big|X^{\a,t,\o}_{r\we (t'-t)}\big|^p &\le & \hat C (t'-t)^{\frac{p}{2}}
\eeaa
(because $\si$ is bounded), we see that
 \bea\label{fix_omega}
 |u(t,\o) - u(t',\o) | 
 &\le& 
 \rho\Big(\big(1+\tilde C^{\frac{1}{2p}}\big) \big((T+1)\hat C\big)^{\frac{1}{p}} (t'-t)^\frac12\Big).
 \eea
The required result \eqref{ud1unifcont} is now a direct consequence of \eqref{fix_t} and \eqref{fix_omega}.
\end{eg}

\end{document}